\documentclass[12pt,oneside]{amsart}
\usepackage[utf8]{inputenc}
\usepackage[english]{babel}
 \usepackage{dutchcal} % \mathcal: also small letters. \mathbcal = bold ones. 
 
\usepackage[margin=1in]{geometry}

\usepackage{multicol}

\usepackage{amsmath}
\usepackage{amsthm}
\usepackage{amssymb}
\usepackage{graphicx}
\usepackage{mathrsfs}
\usepackage[colorinlistoftodos]{todonotes}
\usepackage[colorlinks=true, allcolors=black]{hyperref}
\usepackage{comment}
\usetikzlibrary{graphs}
\usepackage{float}
\usepackage{xcolor}
\usepackage{soul}
\usepackage{tikz-cd}

\newcommand{\showcomments}{yes}
\renewcommand{\showcomments}{no}

\newsavebox{\commentbox}
{\ifthenelse{\equal{\showcomments}{yes}}%
{\footnotemark
        \begin{lrbox}{\commentbox}
        \begin{minipage}[t]{1.5in}\raggedright\sffamily\tiny
        \footnotemark[\arabic{footnote}]}
{\begin{lrbox}{\commentbox}}}%
{\ifthenelse{\equal{\showcomments}{yes}}%
{\end{minipage}\end{lrbox}\marginpar{\usebox{\commentbox}}}
{\end{lrbox}}}

{\ifthenelse{\equal{\showrefcomments}{yes}}%
{\footnotemark
        \begin{lrbox}{\commentbox}
        \begin{minipage}[t]{1.25in}\raggedright\sffamily\tiny
        \footnotemark[\arabic{footnote}]}
{\begin{lrbox}{\commentbox}}}%
{\ifthenelse{\equal{\showrefcomments}{yes}}%
{\end{minipage}\end{lrbox}\marginpar{\usebox{\commentbox}}}
{\end{lrbox}}}

\theoremstyle{definition} 
\newtheorem{thm}{Theorem}[section]
\newtheorem{lem}[thm]{Lemma}
\newtheorem{cor}[thm]{Corollary}

\newtheorem{prop}[thm]{Proposition}

\newtheorem{theorem}[thm]{Theorem}

\newtheorem{lemma}[thm]{Lemma}

\newtheorem{conjecture}[thm]{Conjecture}

\theoremstyle{definition}
\newtheorem{defn}[thm]{Definition}
\newtheorem{definition}[thm]{Definition}

\theoremstyle{remark}
\newtheorem{rem}[thm]{Remark}

\newtheorem{nonexmp}[thm]{Non-example}

\newtheorem{remark}[thm]{Remark}

\newtheorem{construction}[thm]{Construction}

\newcommand{\nclose}[1]{\ensuremath{\langle\!\langle#1\rangle\!\rangle}}
\newcommand{\dist}{\textup{\textsf{d}}}

\DeclareMathOperator{\diam}{\text{diam}}

\newcommand{\field}[1]{\mathbb{#1}}
\newcommand{\integers}{\ensuremath{\field{Z}}}

\newcommand{\naturals}{\ensuremath{\field{N}}}

\newcommand{\euler}{\chi}

\newcommand{\C}{\mathcal{C}}
\newcommand{\Hem}{\mathcal{H}}
\newcommand{\neb}{\mathcal N}

\DeclareMathOperator{\rank}{rank}
\DeclareMathOperator{\sys}{\mathrm sys}

\DeclareMathOperator{\Aut}{Aut}
\DeclareMathOperator{\Stab}{Stab}
\DeclareMathOperator{\Comp}{Comp}

\title[Small-cancellation quotients of non-products]
{Cubical small-cancellation quotients of non-products}
\author{Macarena Arenas}
\author{Kasia Jankiewicz}
\author{Daniel T. Wise}
        \address{Clare College \\
            University of Cambridge \\
            United Kingdom and \\
            Centre for Mathematical Sciences\\
			University of Cambridge\\
			United Kingdom}
	\email{mcr59@cam.ac.uk}
        \address{Department of Mathematics\\
			University of California\\
			Santa Cruz, CA, USA}
	\curraddr{Department of Mathematics\\
			University of British Columbia\\
			Vancouver, British Columbia, Canada}
	\email{kasia@math.ubc.ca}
	\address{Department of Mathematics \& Statistics\\
                    McGill University \\
                    Montreal, Quebec, Canada}
        \curraddr{Faculty of Mathematics and Computer Science\\
                    Weizmann Institute of Science  \\
                    Rehovot, Israel}
        \email{daniel.wise@weizmann.ac.il}

\begin{document}

\begin{abstract}
We show that if $X$ is a compact nonpositively curved cube complex which is a non-product, then $\pi_1 X$  has a proper quotient that is $\pi_1$ of a cube  complex with these same properties.

\end{abstract}
\setcounter{tocdepth}{1}
\maketitle

\vspace{-15pt}
\tableofcontents
\vspace{-15pt}

%%%%%%%%%%%%%%%%%%%%%%%%%%%%%%%%%%%%%%%%%%%%%%%%%%%%%%%%%%%%%%%%%%%%%
%%%%%%%%%%%%%%%%%%%%%%%%%%%%%%%%%%%%%%%%%%%%%%%%%%%%%%%%%%%%%%%%%%%%%
\section{Introduction}
%%%%%%%%%%%%%%%%%%%%%%%%%%%%%%%%%%%%%%%%%%%%%%%%%%%%%%%%%%%%%%%%%%%%%
%%%%%%%%%%%%%%%%%%%%%%%%%%%%%%%%%%%%%%%%%%%%%%%%%%%%%%%%%%%%%%%%%%%%%

Burger and Mozes~\cite{BurgerMozes97} constructed nonpositively curved cube complexes $X$ with infinite simple $\pi_1X$.
In that case $\widetilde X$ is a product of two trees.
A nonpositively curved cube complex $X$ is a \emph{non-product} if
$\pi_1X$ is nontrivial, non-cyclic, and the universal cover $\widetilde X$ does not contain a $\pi_1X$-invariant convex subcomplex isomorphic to a nontrivial product of unbounded cube complexes.

We obtain the following attractive counterpoint to the Burger-Mozes result:

\begin{theorem}\label{thm: main intro cocompact}
[Cubically non-simple]
Let $G=\pi_1X$  where $X$ is a finite dimensional non-product nonpositively curved cube complex. Then $G$ has a nontrivial proper quotient $\bar G=\pi_1 \bar X$ with $\bar X$ a nonproduct nonpositively curved cube complex.
Moreover, if $X$ is compact, then $\bar X$ is compact.
\end{theorem}

\begin{remark}
In Theorem~\ref{thm: main intro cocompact},
for any finite set $S \subset G$, the quotient $\bar G$ can be chosen so that the image of  every nontrivial element in $S$ is nontrivial in $\bar G$. 
\end{remark}

This result raises the following possibility, which is out of reach with current technology in general. Our work resolves it in the cubical setting.

\begin{conjecture}\label{conj: CAT(0)}
    Every CAT(0) group with a rank one element has a proper quotient with the same properties.
\end{conjecture}

To prove Theorem~\ref{thm: main intro cocompact}, we work in the context of cubical small-cancellation theory. A cubical presentation $\langle X \mid \{Y_i \rightarrow X\} \rangle$ consists of   a nonpositively curved cube complex $X$ and a collection of local isometries $Y_i \rightarrow X$. This data determines a quotient $\pi_1X /\nclose{\pi_1Y_i}$ that arises as $\pi_1$ of the space $X^*$ obtained by coning-off each $Y_i$ in $X$. Cubical presentations  are natural generalizations of  group presentations, and are used to study  quotients of  cubulated groups, similar to how ordinary presentations are used to study quotients of free groups. 
\emph{Classical} small-cancellation-theory yields tractable quotients of free groups, and \emph{cubical} small-cancellation theory is the corresponding framework in the cubical setting. See Section~\ref{sec: cubical small cancellation} for the relevant definitions. 
Cubical presentations and cubical small-cancellation theory were introduced in~\cite{WiseBook}. They played an essential role in the proof of the
Malnormal Special Quotient Theorem, and  thus in
the proofs of the Virtual Haken and Virtual Fibering conjectures~\cite{AgolGrovesManning2012}. 
Subsequently, cubical small-cancellation theory has been further studied and utilized 
~\cite{Jankiewicz2017, ArzhantsevaHagen16, JankiewiczWise18, Arenas24IMRN, FuterWise, Arenas24asph, HuangWiseSpecial, ArenasAGT, ArenasJankiewiczWise}.

In the course of this investigation, we found \emph{pseudographs} to be a unifying theme. These are nonpositively curved cube complexes whose hyperplanes are contractible (see Definition~\ref{defn:pseudograph}). They arise naturally in the work of Caprace and Sageev on rank rigidity, and they are readily used as relators in a cubical small-cancellation quotient.
The flow of ideas in the text is then:\vspace{10pt}

\begin{center}
$X$ is a finite-dimensional non-product\\ 
$\Downarrow$\\
$X$ contains a rank~$2$ pseudograph \\
$\Downarrow$\\
$X$ has a rich family of cubical small-cancellation quotients.\\    
\end{center}

\vspace{10pt}
In view of the above, Theorem~\ref{thm: main intro cocompact} is a consequence of the following.

\begin{thm}\label{thm:main}[Theorem~\ref{thm: main without cubulation} and Theorem~\ref{thm:main with all properties}]
    Let $X$ be a
    nonpositively curved cube complex that admits a local isometry of a rank~$2$ superconvex pseudograph. 
    For every $\alpha \leq \frac 1{16}$ 
    there exists a pseudograph $Y\looparrowright X$ with $\pi_1 Y\neq 1$ such that $\langle X\mid Y\rangle$ is a cubical $C'(\alpha)$ small-cancellation presentation.    Moreover, we can choose $Y$ so that $\pi_1 X / \nclose{ \pi_1 Y}$ acts freely on a CAT(0) cube complex $\mathcal{C}$.

Furthermore, if $X$ is  compact, then
  $Y$ can be chosen so that $\pi_1 X / \nclose{\pi_1 Y}$   acts properly and cocompactly on   $\mathcal{C}$.
\end{thm}

%%%%%%%%%%%%%%%%%%%%%%%%%%%%%%%%%%%%%%%%%%%%%%%%%%%%%%%%%%%%%%%%%%%%%
\subsection*{Organization}

In Section~\ref{sec: cubes} we discuss various results and notions  related to rank rigidity, and use them to give a characterization of non-product nonpositively curved cube complexes. 
In Section~\ref{sec:pseudograph} we introduce pseudographs, and explain the relation between pseudographs and non-products.
In Section~\ref{sec: cubical small cancellation} we review the  cubical small-cancellation terminology. 
In Section~\ref{sec: small cancellation via pseudograph} we explain how to use pseudographs to produce cubical small-cancellation quotients of non-products. 
In Section~\ref{sec: disc diagrams} we review disc diagrams, state the cubical versions of Greendlinger's Lemma and the Ladder Theorem, and describe the  more sophisticated $B(6)$ cubical small-cancellation condition and related results that lead to cubulation. 
In Section~\ref{sec: wall traingles} we prove that certain disc diagrams arising from collections of intersecting walls in a $B(6)$ presentation are uniformly thin. 
In Section~\ref{sec: cubulating} we begin the proof of Theorem~\ref{thm:main}. Namely, $\pi_1$ of a suitably chosen $B(6)$ cubical presentation $X^*$ acts on the non-product CAT(0) cube complex dual to the $B(6)$ wallspace structure on $\widetilde X^*$.  
In Section~\ref{sec:freeness} we show that $X^*$  can moreover be chosen so that the action is free. 
In Section~\ref{sec:cocompact} we complete the proof of Theorem~\ref{thm:main} by verifying cocompactness of the action on the dual cube complex obtained in Section~\ref{sec: cubulating}.

\subsection*{Related work}
It is known that for any non-product $X$, there exists a nontrivial quotient of $\pi_1 X$, as $\pi_1 X$ is acylindrically hyperbolic. Indeed,
Caprace and Sageev~\cite{CapraceSageev2011} showed that every group acting properly and cocompactly on a geodesically complete non-product CAT(0) cube complex contains a rank~one isometry. Osin showed every non virtually cyclic group acting properly on a proper CAT(0) space and containing a rank one isometry is acylindrically hyperbolic \cite{Osin2016}. 
In particular, the action of $\pi_1 X$ on its contact graph, which is hyperbolic, is acylindrical  when $X$ admits a factor system \cite{BehrstockHagenSisto}. Acylindrically hyperbolic groups have many proper quotients, e.g.\ they are SQ-universal \cite{DahmaniGuirardelOsin2017}. There exists a version of small-cancellation theory for acylindrical hyperbolic groups which also produces nontrivial quotients \cite{Hull2016}.
Acylindrical hyperbolicity has been applied to cube complexes to show that the growth-rate of $\pi_1X$ is strictly larger than the growth-rate of its hyperplane stabilizers \cite{DahmaniFuterWise}.
However, none of these results provide cubical geometry for the quotient group.

\subsection*{Acknowledgements}
The first author was supported by the Denman Baynes Research Fellowship at Clare College, Cambridge.
The second author was supported by the National Science Foundation
under Grants No. DMS-1926686 and DMS-2238198. The third author was supported by NSERC.

\section{CAT(0) cube complexes}\label{sec: cubes}
%%%%%%%%%%%%%%%%%%%%%%%%%%%%%%%%%%%%%%%%%%%%%%%%%%%%%%%%%%%%%%%%%%%%%
%%%%%%%%%%%%%%%%%%%%%%%%%%%%%%%%%%%%%%%%%%%%%%%%%%%%%%%%%%%%%%%%%%%%%
We  assume familiarity with the basic background on CAT(0) cube complexes, for which the reader can consult~\cite{GGTbook14,  WiseCBMS2012, WiseBook}, for instance.

\subsection{Halfspaces}
A \emph{hyperplane} $U$ in a CAT(0) cube complex $\widetilde X$ is a non-empty connected subspace such that $U\cap c$ is a midcube of $c$ for each cube $c$ of $\widetilde X$. Its two associated \emph{halfspaces} are the closures of the components of  $\widetilde X-U$ and are denoted $U^+, U^-$. While halfspaces are not subcomplexes of $\widetilde X$, we can define associated subcomplexes as follows. 
The \emph{carrier} $N(U)$ of a hyperplane $U$ is the convex subcomplex consisting of the union of all cubes intersecting $U$.
A \emph{minor halfspace} is the closure of a component of the complement of the interior of $N(U)$. A \emph{major halfspace} is the closure of the complement of a minor halfspace.

In the next few sections we describe some background, mostly from Caprace-Sageev \cite{CapraceSageev2011}.

\subsection{Properties of actions}
Let $G$ act on a CAT(0) cube complex $\widetilde X$. A hyperplane $U\subseteq \widetilde X$ is $G$-essential for any (equivalently every) basepoint $v\in \widetilde X$, each halfspace of $U$ contains points of $Gv$ lying arbitrarily far from $U$. 
The action of $G$ on $\widetilde X$ is \emph{essential} if each hyperplane of $\widetilde X$ is $G$-essential. 
The \emph{$G$-essential core} of $\widetilde X$ is the cube complex dual to all the $G$-essential hyperplanes of $\widetilde X$. The action of $G$ on its $G$-essential core is essential. 
When $G$ either acts on $\widetilde X$ cocompactly or without fixed point at infinity, then the $G$-essential core of $\widetilde X$ is unbounded if and only if $G$ has no fixed points in $\widetilde X$. In that case the $G$-essential core embeds as a convex $G$-invariant subcomplex of the cubical subdivision of $\widetilde X$ \cite[Prop~3.5]{CapraceSageev2011}. Thus, in the setting of Theorem~\ref{thm: main intro cocompact}, up to possibly passing from $\widetilde X$ to its $\pi_1X$-essential core, we can assume that the action of $\pi_1X$ on $\widetilde X$ is essential. Indeed, if $X$ is a non-product then so is the $\pi_1X$-essential core of $\widetilde X$. It is worth noting that when $\pi_1X$ is not trivial or cyclic, the $\pi_1 X$-essential core is a non-product if and only if it is not a direct product of unbounded cube complexes.

%\subsection{No fixed point at infinity}
 A group $G$ acts on a CAT(0) space $\widetilde X$ \emph{without a fixed point at infinity}, if there does not exist $\xi \in \partial X$ where $\partial X$ is the visual boundary of $X$, with $g\xi =\xi$ for all $g\in G$.
By \cite[Thm~1.9 and Clm~4.9]{Genevois25}, if $G$ acts on a finite dimensional CAT(0) cube complex $\widetilde X$, and for every finite index subgroup $H\subseteq G$ the commutator subgroup $[H,H]$ is infinite, then $G$ acts on $\widetilde X$ without a fixed point at infinity.

\begin{remark}\label{rmk: no fixed point}
    When $X$ is a finite-dimensional non-product, then $G=\pi_1X$ satisfies those assumptions. Indeed, otherwise $G$ is virtually finite-by-abelian, so in particular virtually abelian, but this contradicts the assumption that $G$ is the $\pi_1$ of a non-product.
\end{remark}

By Remark~\ref{rmk: no fixed point}, in the setting of Theorem~\ref{thm: main intro cocompact} the action of $\pi_1X$ on $\widetilde X$ has no fixed point at infinity. Consequently, we can apply Proposition~\ref{prop:the core} to deduce Theorem~\ref{thm: main intro cocompact} from Theorem~\ref{thm:main}.

\subsection{Characterization of non-products}

Two hyperplanes $U_1, U_2$ are \emph{strongly separated} if there does not exist a hyperplane $V$ which intersects both $U_1, U_2$.

\begin{prop}[{\cite[Prop~5.1]{CapraceSageev2011}}]\label{prop: non-product characterization}
    Let $\widetilde X$ be a finite-dimensional unbounded CAT(0) cube complex such that $\Aut(\widetilde X)$ acts essentially without a fixed point at infinity. Then the following conditions are equivalent:
    \begin{enumerate}
        \item $\widetilde X$ is a non-product.
        \item There is a pair of strongly separated hyperplanes.
        \item For each halfspace $U^+$ there is a pair of halfspaces $U_1,U_2$ such that $U_1^+\subseteq U^+\subseteq U_2^+$ and the hyperplanes $U_1, U_2$ are strongly separated.
    \end{enumerate}
\end{prop}

\begin{lem}[Flipping Lemma {\cite[Thm~4.1]{CapraceSageev2011}}] \label{lem: flipping lemma}
      Let $\widetilde X$ be a finite-dimensional CAT(0) cube complex, and let $G\subseteq \Aut(\widetilde X)$ act essentially, without a fixed point at infinity. For each halfspace $U^+$, there exists $\gamma\in G$ with $U^-\subsetneq \gamma U^+$.
\end{lem}

\begin{lem}[Double Skewering Lemma {\cite{CapraceSageev2011}}] \label{lem: double skewering lemma}
      Let $\widetilde X$ be a finite-dimensional CAT(0) cube complex and $G \subseteq \Aut(\widetilde X)$ be a group acting essentially without a fixed point at infinity. Then for any two half-spaces $U_1^+\subsetneq U_2^+$, there exists $g\in G$ such that $U_2^+\subsetneq g U_1^+$.
\end{lem}

A hyperplane $U$ \emph{separates} sets $A,B\subseteq \widetilde X$ if $A\subseteq U^+$ and $B\subseteq U^-$, or vice-versa.
A \emph{facing triple of hyperplanes} is a collection of three disjoint hyperplanes $U_1, U_2, U_3$ such that no $U_i$ separates the other two hyperplanes. In particular, we can always choose the halfspaces $U_1^-, U_2^-, U_3^-$ of these hyperplanes so that the intersection $U_1^-\cap U_2^-\cap U_3^-$ is non-empty. 
A \emph{wedge} corresponding to a facing triple $U_1, U_2, U_3$ in $\widetilde X$ is the intersection of the associated minor halfspaces of a facing triple, i.e.\ a maximal subcomplex contained $U_1^-\cap U_2^-\cap U_3^-$.

\begin{thm}[{\cite[Thm~7.2]{CapraceSageev2011}}] \label{thm: no facing triple} 
    Let $\widetilde X$ be a finite-dimensional CAT(0) cube complex such that $\Aut(\widetilde X)$ acts essentially and satisfies at least one of the following conditions:
    \begin{enumerate}
        \item $\Aut(\widetilde X)$ has no fixed points at infinity.
        \item $\Aut(\widetilde X)$ acts cocompactly and $X$ is locally compact.
    \end{enumerate}
    Then $\Aut(\widetilde X)$ stabilizes some Euclidean flat if and only if there is no facing triple of hyperplanes in $\widetilde X$.
\end{thm}

\begin{prop}\label{prop: strongly separated facing triple}
     Let $\widetilde X$ be a finite-dimensional CAT(0) cube complex, and let $G\subseteq \Aut(\widetilde X)$ act essentially, without a fixed point at infinity.
     If $\widetilde X$ is a non-product, then there exists a facing triple of pairwise strongly separated hyperplanes in $\widetilde X$.
\end{prop}

\begin{proof}
       By Theorem~\ref{thm: no facing triple} there exists a facing triple $U_1, U_2, U_3$ of hyperplanes in $\widetilde X$. Let $U_i^+$ be the halfspace of $U_i$ which does not contain the other two hyperplanes. By Proposition~\ref{prop: non-product characterization} there exists a pair of strongly separated hyperplanes $V_{1}, W_{1}$ such that $V_1^+\subseteq U_1^+\subseteq W_1^+$. 

       First consider the case where $U_2,U_3\subseteq W_1^-$, as in the left of Figure~\ref{fig: constructing strongly separated facing triple}. 
       \begin{figure}
           \centering
           \includegraphics[scale =0.7]{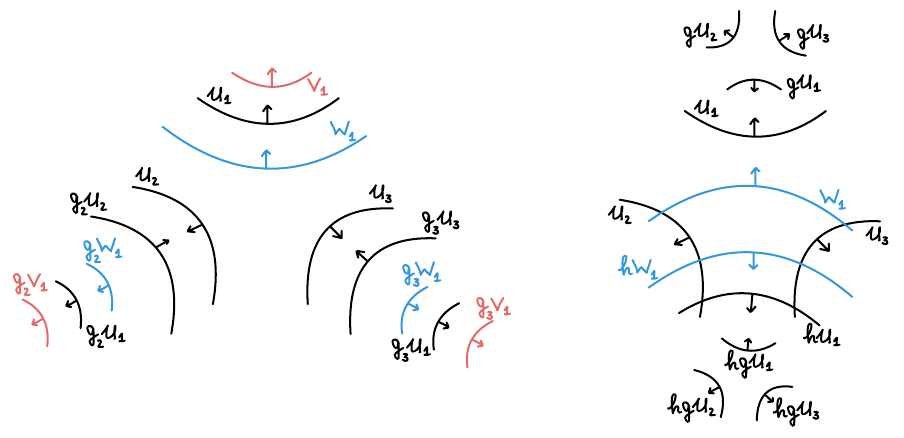}
           \caption{Left: Strongly separated facing triple $\{V_1, g_2V_1, g_3V_1\}$, when $U_2, U_3\subseteq W_1^-$. Right: The facing triple $\{U_1, hgU_2, hgU_3\}$ plays the role of $\{U_1, U_2, U_3\}$  on the left.}
           \label{fig: constructing strongly separated facing triple}
       \end{figure}
       For $i=2,3$, let $g_i\in G$ be an element flipping $U_1$, provided by Lemma~\ref{lem: flipping lemma}, i.e.\ $g_iU_i^-\subseteq U_i^+$. We claim that $\{V_1, g_2V_1, g_3V_1\}$ is a strongly separated facing triple. For symmetry, let $g_1 = id$, so we can write $V_1 = g_1V_1$. Then for $\{i,j\}\subseteq \{1,2,3\}$ we have
       \[g_iV_1^+ \subseteq g_iW_1^+\subseteq g_iU_i^-\subseteq g_jU_j^+\subseteq g_jW_j^-\subseteq g_jV_j^-.\] This shows that $g_iV_i$ and $g_jV_j$ are strongly separated, since $g_iV_1$ and $g_1W_1$ are.
      
       Now consider the remaining case, see the right diagram in Figure~\ref{fig:  constructing strongly separated facing triple}. 
       We will construct a new triple of hyperplanes $\{U_1, U_2', U_3'\}$ where $U_2', U_3'\subseteq W_1^-$.
       First let $g\in G$ be an element that flips $U_1$, i.e.\ $gU_1^-\subseteq U_1^+$. We also have $gU_i^+\subseteq gU_1^-\subseteq U_1^+$ for $i=2,3$. Now let $h\in G$ be an element flipping $W_1$, i.e.\ $hW_1^+\subseteq W_1^-$. We claim that $\{U_1, hgU_2, hgU_3\}$ is a facing triple of hyperplanes such that $hgU_i^+\subseteq W_1^-$. Clearly $hgU_2, hgU_3$ are disjoint as images of disjoint hyperplanes $U_2,U_3$ under $hg$. Moreover, for $i=2,3$ we have $hgU_i^+\subseteq hgU_1^-\subseteq hU_1^+\subseteq hW_1^+\subseteq W_1^-\subseteq U_1^-$.
\end{proof}

To summarize, we have the following.

\begin{cor}\label{cor: facing triple implies strongly separated facing triple}
    Let $X$ be a finite-dimensional nonpositively curved cube complex such that $\pi_1 X$ is nontrivial and acts essentially without a fixed point at infinity. The following are equivalent.
    \begin{enumerate}
        \item\label{cond1} $X$ is a non-product,
        %\item there exists a facing triple of hyperplanes in $\widetilde X$,
        \item\label{cond2} there is pair of strongly separated hyperplanes in $\widetilde X$.
        \item\label{cond3} there is a facing triple of pairwise strongly separated hyperplanes in $\widetilde X$.
        
    \end{enumerate}
\end{cor}
\begin{proof}
    The implication $(\ref{cond3})\to (\ref{cond2})$ is obvious, the implication $(\ref{cond1})\to (\ref{cond3})$ holds by %Theorem~\ref{thm: no facing triple} and    
Proposition~\ref{prop: strongly separated facing triple}, and the implication $(\ref{cond2})\to (\ref{cond1})$ is part of Proposition~\ref{prop: non-product characterization}.
    \end{proof}

%%%%%%%%%%%%%%%%%%%%%%%%%%%%%%%%%%%%%%%%%%%%%%%%%%%%%%%%%%%%%%%%%%%%%
%%%%%%%%%%%%%%%%%%%%%%%%%%%%%%%%%%%%%%%%%%%%%%%%%%%%%%%%%%%%%%%%%%%%%

\section{Superconvex pseudographs}\label{sec:pseudograph}
%%%%%%%%%%%%%%%%%%%%%%%%%%%%%%%%%%%%%%%%%%%%%%%%%%%%%%%%%%%%%%%%%%%%%
%%%%%%%%%%%%%%%%%%%%%%%%%%%%%%%%%%%%%%%%%%%%%%%%%%%%%%%%%%%%%%%%%%%%%

A map $\varphi:Y \rightarrow X$ between
nonpositively curved cube complexes is a \emph{local isometry} if $\varphi$ is locally injective, maps open cubes homeomorphically to open cubes, and whenever $a$, $b$ are concatenable edges of $Y$, if $\varphi(a)\varphi(b)$ is a subpath of the attaching map of a $2$-cube of $X$, then $ab$ is a subpath of a $2$-cube in $Y$. If $Y\to X$ is a local isometry, then its lift $\widetilde Y \to \widetilde X$ is an embedding.

\begin{defn}[Superconvex Pseudograph]\label{defn:pseudograph}
A \emph{pseudograph} 
is a nonpositively curved cube complex whose immersed hyperplanes are contractible. 
Note that $\pi_1Y$ is free when $Y$ is a compact pseudograph \cite[Lem~9.9]{WiseBook}.
A compact pseudograph $Y$ has \emph{rank~$n$} if $\pi_1Y$ is a free group of rank~$n$. We will always assume that pseudographs are compact without stating it explicitly.

A local isometry $Y\rightarrow X$ is \emph{superconvex} if  $\widetilde Y$ is convex and  for any bi-infinite geodesic $\tilde \gamma$ in $\widetilde X$, if $\tilde \gamma \subset \neb_r(\widetilde Y)$ for some $r >0$ then $\tilde \gamma\subset \widetilde Y$, where $\neb_r(\widetilde Y)$ denotes the closed $r$-neighborhood of $\widetilde Y$ in $\widetilde X$. 
When $X$ is locally-finite and $Y$ is compact, this is equivalent to having an upper bound on the length $\ell$ of a combinatorial strip $[0,1]\times[0,\ell]\subseteq \widetilde X$ such that $\{0\}\times [0,\ell]\subseteq \widetilde Y$ but $[0,1]\times[0,\ell]\not\subseteq \widetilde Y$. See~\cite[Lem~2.40]{WiseBook}.
\end{defn}

\begin{nonexmp}\label{nonex: not superconvex}
    Let $X$ be a product of  non-tree finite graphs $Y_1\times Y_2$ with the natural cubical structure. Then $Y_i\to X$ is a local isometry of a pseudograph but  is not superconvex.
\end{nonexmp}

\begin{remark}\label{rem: superconvexity} Given a superconvex rank~$2$ pseudograph $Y\to X$, any local isometry $W \rightarrow Y$ provides a superconvex pseudograph $W \rightarrow X$. See Lemma~\ref{lem:bounded wall pieces}. 
\end{remark}

The goal of this section is the following:

\begin{prop}\label{prop:the core}
Let $X$ be a finite-dimensional nonpositively curved cube complex, where $\pi_1 X$ is nontrivial and acts essentially, without a fixed point at infinity. The following statements are equivalent:
\begin{enumerate}
\item\label{cond:nonproduct} $X$ is a non-product.
\item\label{cond:strongly separated pair} $\widetilde X$ has a pair of strongly separated hyperplanes.
\item\label{cond:facing triple} $\widetilde X$ has a facing triple of strongly separated hyperplanes.
\item\label{cond:rank1 pseudograph} There is a local isometry $Y\rightarrow X$ of a superconvex rank~$1$ pseudograph.
\item\label{cond:rank2 pseudograph} There is a local isometry $Y\rightarrow X$ of a superconvex rank~$2$ pseudograph.
\end{enumerate}
\end{prop}

The equivalence of \eqref{cond:nonproduct}, \eqref{cond:strongly separated pair}, and \eqref{cond:facing triple} is Corollary~\ref{cor: facing triple implies strongly separated facing triple}. We will prove $\eqref{cond:facing triple}\Rightarrow \eqref{cond:rank2 pseudograph}$ and $\eqref{cond:rank1 pseudograph}\Rightarrow\eqref{cond:strongly separated pair}$. The implication $\eqref{cond:rank2 pseudograph}\Rightarrow\eqref{cond:rank1 pseudograph}$ uses Remark~\ref{rem: superconvexity} and the following observation: if $Y$ is a compact cube complex and $\pi_1 Y$ is hyperbolic, then every quasiconvex, in particular cyclic, subgroup of $\pi_1 Y$ has a cocompact core in $Y$ \cite{HaglundGraphProduct}.

\begin{proof}[Proof of Proposition~\ref{prop:the core} \eqref{cond:facing triple} $\Rightarrow$ \eqref{cond:rank2 pseudograph}]
Let $\{U_1, U_2, U_3\}$ be a strongly separated facing triple of hyperplanes in $\widetilde X$, and let $W$ be the associated wedge contained in $U_1^-\cap U_2^-\cap U_3^-$.

By Lemma~\ref{lem: flipping lemma} applied to $U_3^+$ there exists $g\in G$ such that $U_3^- \subseteq g U_3^+$. Let $U_4 =g U_1$ and $U_5 = g U_2$. See Figure~\ref{fig: free group}. Note that $\{U_4, U_5, U_3\}$ is also a strongly separated facing triple of hyperplanes, and let $W'$ be the associated wedge contained in $U_4^-\cap U_5^-\cap U_3^+$. 
By applying Lemma~\ref{lem: double skewering lemma} to the pairs of halfspaces $U_1^+\subseteq U_4^-$ and $U_2^+\subseteq U_5^-$, we obtain two isometries $a,b$ in $G$, such that $aU_4^-\subseteq U_1^+$ and $bU_5^-\subseteq U_2^+$. See Figure~\ref{fig: free group}.
\begin{figure}
    \centering
    \includegraphics[width=0.5\linewidth]{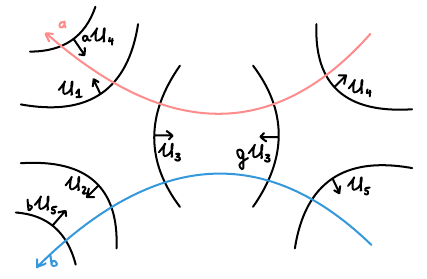}
    \caption{Isometries $a,b$ form a free group.}
    \label{fig: free group}
\end{figure}
By the ping-pong lemma, $F=\langle a,b\rangle$ is a  rank~$2$ free subgroup of $G$.

Let $p,q$ be points in the wedges $W, W'$ respectively, and let $A_{p,q}$ be a geodesic from $p$ to $q$. Similarly, let $A_{p, aq}$ and $A_{p, bq}$ be geodesics from $p$ to $aq$ and $bq$ respectively.
Let $T$ be an $F$-cocompact tree that is the union of all the $F$-translates of  $A_{p,q}$, $A_{p,aq}$, and $A_{p,bq}$.

Let $\widetilde Y$ be the convex hull of $T$.
Note that $\widetilde Y$ is the intersection of minor half-spaces of $\widetilde X$ that contain $T$.
It is clear that $\widetilde Y$ is $F$-invariant.
We claim that $\widetilde Y$ is $F$-cocompact.
Firstly, note that any hyperplane of $\widetilde Y$ is dual to an edge of $T$,
as otherwise, $T$ would lie in one of its associated minor halfspaces.
Secondly, each hyperplane of $\widetilde Y$ might intersect at most one $F$-translate of the wedge $W'$ as otherwise it would have two intersect translates (by the same element from $G$) of more than one of the hyperplanes $U_1, U_2, U_3$, violating their strong separation.
Let $r$ be the maximal length of $A_{p,q}, A_{p,aq}, A_{p,bq}$.
Observe that any collection of pairwise-crossing hyperplanes of $\widetilde Y$
are dual to edges that lie in a $2r$-ball in $T$.
Thus $\widetilde Y$ has finitely many $F$-orbits of maximal cubes. Consequently $\widetilde Y$ is $F$-cocompact. Let $Y = \widetilde Y/F$, then $Y\to X$ is a local isometry.
We now claim that $Y\to X$ is superconvex. An upper bound on the length $\ell$ of a rectangle $[0,1]\times[0,\ell]$ such that $\{0\}\times [0,\ell]\subset \widetilde Y$ but $[0,1]\times[0,\ell]\not\subset \widetilde Y$ can be taken to be $\ell = 2r$, again by the strong separation of $U_1, U_2, U_3$.
Finally, strong separation of $U_1, U_2, U_3$ also ensures that all the hyperplanes of $\widehat Y$ are compact. In particular, the hyperplanes of $Y$ must be simply connected, and therefore contractible. This proves that $Y$ is a pseudograph.
\end{proof}

The  proof of   \eqref{cond:rank1 pseudograph} $\Rightarrow$ \eqref{cond:strongly separated pair} uses the Helly property of CAT(0) cube complexes. 
This was originally stated in \cite{Gerasimov98},
 but see, for instance, \cite[Lem~2.10]{WiseBook}.
\begin{lem}[Helly property]\label{lem:helly}
Let $\{Y_i\}$ be a finite collection of convex subcomplexes of a CAT(0) cube complex. If $Y_i\cap Y_j\neq \emptyset$ for each $i,j$, then $\bigcap_i Y_i \neq \emptyset$.
\end{lem}

\begin{proof}[Proof of Proposition~\ref{prop:the core} \eqref{cond:rank1 pseudograph} $\Rightarrow$ \eqref{cond:strongly separated pair}]
 Let $Y \rightarrow X$ be a superconvex rank~$1$ pseudograph. Since the immersed hyperplanes of $Y$ are contractible and compact, their diameter is bounded by some constant $K$. This  also bounds  the diameters of  hyperplanes of $\widetilde Y$.
 Let $L$ be the upper bound from the definition of superconvexity, i.e.\ for every $\ell>L$ if $[0,1]\times[0,\ell]\subseteq \widetilde X$ and  $\{0\}\times[0,\ell]\subseteq \widetilde Y$, then 
$[0,1]\times[0,\ell]\subseteq \widetilde Y$. We set $M = \max\{K,L\}$.
By Lemma~\ref{lem:helly}, disjoint hyperplanes of $\widetilde Y$ extend to disjoint hyperplanes in $\widetilde X$. 
Thus any two hyperplanes in $\widetilde Y$ that are at distance greater than $M$ in $\widetilde Y$ are strongly separated in $\widetilde X$.
\end{proof}

The following direct proof of $\eqref{cond:rank2 pseudograph} \Rightarrow \eqref{cond:facing triple}$ is illustrative.

\begin{proof}[Proof of Proposition~\ref{prop:the core} \eqref{cond:rank2 pseudograph} $\Rightarrow$ \eqref{cond:facing triple}]
 Let $Y \rightarrow X$ be a superconvex rank~$2$ pseudograph. Let $K$ be the upper bound on  the diameters of  hyperplanes of $\widetilde Y$.
 Let $L$ be the upper bound from the definition of superconvexity. We set $M = \max\{K,L\}$.
 
By Lemma~\ref{lem:helly}, disjoint hyperplanes of $\widetilde Y$ extend to disjoint hyperplanes in $\widetilde X$. 
There exists a finite cover of $Y$ where all hyperplanes are embedded. By possibly passing to a superconvex subcomplex of that finite cover, we can assume $Y$ is a rank~$2$ superconvex pseudograph with embedded hyperplanes.

Let $U$ be a non-separating hyperplane of $Y$, let $Y'=Y-N(U)$ and note that $\pi_1Y'\cong \integers$.
Consider a cyclic cover $\widehat Y'$ of $Y'$ whose degree $m$ is sufficiently large, and let $N_0$ be a lift of $N(U)$.

We claim there is a non-separating hyperplane $U'$  of $\widehat Y'$ (i.e.\ so that  $\pi_1(\widehat Y'-U')=1$) such that $\dist_{\widetilde Y}(U', N_0)>M$. Indeed $Y'$ and therefore also the universal cover $\widetilde Y'$ are locally finite, but the diameter of $\widetilde Y'$ is infinite. Thus if the degree $m>M$, then there exist hyperplanes at distance greater than $M$ away from $N_0$, any such non-separating hyperplane can be picked as $U'$. 
In particular $U'$ does not intersect $N_0$.

Let $\widehat Y$ be the cover of $Y$ corresponding to $\widehat Y'$, i.e.\ so that $\Aut(\widehat Y\to Y)$ is naturally isomorphic to $\Aut(\widehat Y'\to Y')$. 
Consider lifts of $U,U'$ to $\widehat Y$, where $\partial N(U)$ naturally corresponds to $N_0$.
Then the lift of $\widehat Y-U-U'$ to $\widetilde X$ provides a facing quadruple of hyperplanes in $\widetilde X$. Indeed, $\widehat Y-U-U'$ embeds in $\widetilde Y$ and is bounded by a quadruple of disjoint hyperplanes $U_1, U_2, U'_1, U'_2$. These hyperplanes extend to disjoint hyperplanes of $\widetilde X$, which we continue to denote by $U_1, U_2, U'_1, U'_2$. We now show that $U_1, U'_1, U'_2$ are pairwise strongly separated.  Note that the distance between any two of $U_1, U'_1, U'_2$ is bounded below by $M$.

Let $V$ and $V'$ be two hyperplanes that intersect $\widetilde Y$.
We claim that if there exists a hyperplane $W$ in $\widetilde X$ that intersects both $V$ and $V'$, then $V$ and $V'$ are already intersected by another hyperplane in $\widetilde Y$.

Consider a disc diagram $D$ containing segments of hyperplanes $V,V', W$ whose boundary path is a concatenation of four paths: one lying in $\widetilde Y$, and three paths in the carriers $N(V), N(V'), N(W)$ respectively. We assume that $D$ has minimal area among all the choices of such disc diagrams. By convexity of $N(W), N(V), N(V'), \widetilde Y$ and minimality of $D$, we can deduce that $D$ is a grid. Indeed any hyperplane in $D$ starting in $\widetilde Y$ must exit $D$ in $N(W)$, and vice-versa, as otherwise we can pick a disc diagram of smaller area. This implies that for the hyperplane $W'$ which intersects $V$ and $V$ closest to $\widetilde Y$, the intersection $D\cap N(W')$ is of the form $[0,1]\times [0, \ell]$, where $\{0\}\times [0, \ell]\subseteq \widetilde Y$. By superconvexity of $\widetilde Y$, either $W'$ intersects $V, V'$ in $\widetilde Y$, or $\ell\leq L$.

Thus if $\dist_{\widetilde X}(V, V')> L$, then $V, V'$ are either strongly separated, or they are both intersected by some hyperplane of $\widetilde Y$. Since the diameter of any hyperplane of $\widetilde Y$ is bounded above by $K$, so the hyperplanes $V, V'$ must be strongly separated, if we also have $\dist_{\widetilde X}(V, V')> K$.
\end{proof}

%%%%%%%%%%%%%%%%%%%%%%%%%%%%%%%%%%%%%%%%%%%%%%%%%%%%%%%%%%%%%%%%%%%%%
%%%%%%%%%%%%%%%%%%%%%%%%%%%%%%%%%%%%%%%%%%%%%%%%%%%%%%%%%%%%%%%%%%%%%

\section{Cubical small-cancellation review}\label{sec: cubical small cancellation}
%%%%%%%%%%%%%%%%%%%%%%%%%%%%%%%%%%%%%%%%%%%%%%%%%%%%%%%%%%%%%%%%%%%%%
%%%%%%%%%%%%%%%%%%%%%%%%%%%%%%%%%%%%%%%%%%%%%%%%%%%%%%%%%%%%%%%%%%%%%

%%%%%%%%%%%%%%%%%%%%%%%%%%%%%%%%%%%%%%%%%%%%%%%%%%%%%%%%%%%%%%%%%%%%%
\subsection{Cubical presentations and small-cancellation conditions}
%%%%%%%%%%%%%%%%%%%%%%%%%%%%%%%%%%%%%%%%%%%%%%%%%%%%%%%%%%%%%%%%%%%%%

A \textit{cubical presentation} $\langle X\mid \{Y_i\} \rangle$ consists of the following data:
\begin{enumerate}
    \item A connected nonpositively curved cube complex $X$. 
    \item  A collection of local isometries of connected nonpositively curved cube complexes $Y_i \overset{\varphi_i} \longrightarrow X$.
\end{enumerate} Local isometries of nonpositively curved cube complexes are $\pi_1$-injective, and we define the \emph{fundamental group} of $\langle X\mid \{Y_i\} \rangle$ as $\pi_1 X/\nclose{\{\pi_1 Y_i\}}$. Note that this group is isomorphic to the fundamental group of the space $X^*$ obtained by coning off each $Y_i$ in $X$. Throughout, we write $X^*=\langle X\mid \{Y_i\} \rangle$ and identify a cubical presentation with its coned-off space.
A cubical presentation $X^*$ is \emph{compact} if both $X$ and all $Y_i$'s are compact. The universal cover $\widetilde X^*$ of the coned-off space $X^*$ is $\hat X\cup \bigcup Cone(gY_i)$ where $\hat X$ is the covering space of $X$ associated to the kernel of the map $\pi_1 X \rightarrow \pi_1X^*$, and $g\in \pi_1 X^*$. We refer to $\hat X$ as the \emph{cubical part} of $\widetilde X^*$.

An \emph{abstract contiguous cone-piece} of $Y_j$ in $Y_i$ is an intersection $\widetilde{Y}_j \cap \widetilde{Y}_i$ where either $i \neq j$ or where $i = j$
but $\widetilde{Y}_j \neq \widetilde{Y}_i$. A \emph{cone-piece} of $Y_j$ in $Y_i$ is a path $p \rightarrow P$ in an abstract contiguous cone-piece of $Y_j$ in $Y_i$.
An \emph{abstract contiguous wall-piece} of $Y_i$ is a non-empty intersection $N(H) \cap \widetilde{Y}_i$ where $N(H)$ is
the carrier of a hyperplane $H$ that is disjoint from $\widetilde{Y}_i$. A \emph{wall-piece} of $Y_i$ is a path $p \rightarrow P$ in an abstract contiguous wall-piece of $Y_i$.
A \emph{piece} is either a cone-piece or a wall-piece.

Recall that the \emph{systole} of $Y$, denoted by $\sys(Y)$, is the infimum of lengths of  essential combinatorial paths in $Y$.
For a constant $\alpha > 0$, the cubical presentation $X^*=\langle X \mid \{Y_i\} \rangle$ satisfies the $C'(\alpha)$ \emph{small-cancellation} condition if
$\diam (P) < \alpha \sys(Y_i)$
for every piece $P$ in $Y_i$.

As in the case of classical small-cancellation theory, when $\alpha$ is small, the $C'(\alpha)$ condition provides control over $\pi_1X^*$. This is explained in \cite{WiseBook} at $\alpha=\frac1{12}$, and in \cite{Jankiewicz2017} at $\alpha=\frac18$.

Let $Y \rightarrow X$ be a local isometry. $\Aut_X(Y)$ is the group of automorphisms $\psi: Y \rightarrow Y$ such that the diagram below is commutative:
\[\begin{tikzcd}
Y \arrow[r, "\psi"] \arrow[rd] & Y \arrow[d] \\
                               & X          
\end{tikzcd}\]
If $Y$ is simply connected, then $\Aut_X(Y)$ is equal to $\Stab_{\pi_1X}(Y)$.

We adopt the convention that two elevations $\widetilde{Y}_j, \widetilde{Y}_i$ with $i=j$ are regarded as equal if they differ only by an element of $\Aut_X(Y)$.

%%%%%%%%%%%%%%%%%%%%%%%%%%%%%%%%%%%%%%%%%%%%%%%%%%%%%%%%%%%%%%%%%%%%%

%%%%%%%%%%%%%%%%%%%%%%%%%%%%%%%%%%%%%%%%%%%%%%%%%%%%%%%%%%%%%%%%%%%%%
%%%%%%%%%%%%%%%%%%%%%%%%%%%%%%%%%%%%%%%%%%%%%%%%%%%%%%%%%%%%%%%%%%%%%

\section{Small-cancellation via the pseudograph}\label{sec: small cancellation via pseudograph}
%%%%%%%%%%%%%%%%%%%%%%%%%%%%%%%%%%%%%%%%%%%%%%%%%%%%%%%%%%%%%%%%%%%%%
%%%%%%%%%%%%%%%%%%%%%%%%%%%%%%%%%%%%%%%%%%%%%%%%%%%%%%%%%%%%%%%%%%%%%

In this section we  prove the first part of Theorem~\ref{thm:main}.
Throughout this section we assume  $X$ is a 
nonpositively curved cube complex, which admits a local isometry $Y\to X$ of a superconvex rank~$2$ pseudograph.

The strategy of the proof is as follows. In Section~\ref{sec: passing to malnormal}, we construct a further superconvex rank~$2$ pseudograph whose fundamental group is malnormal in $G$. Then in Section~\ref{sec: induced small-cancellation} we use it to pick a pseudograph so that the cubical presentation with the pseudograph as a relator satisfies the small-cancellation conditions, which also uses the uniform bound on the size of wall-pieces obtained in Section~\ref{sec: uniform bound}. The proof of the first part of Theorem~\ref{thm:main}, combining these ingredients, is in Section~\ref{sec: proof without cocompactness}.

%%%%%%%%%%%%%%%%%%%%%%%%%%%%%%%%%%%%%%%%%%%%%%%%%%%%%%%%%%%%%%%%%%%%%
\subsection{Uniform bound on the size of wall-pieces}\label{sec: uniform bound}
%%%%%%%%%%%%%%%%%%%%%%%%%%%%%%%%%%%%%%%%%%%%%%%%%%%%%%%%%%%%%%%%%%%%%

\begin{lem}\label{lem:bounded wall pieces}
Let $Y \rightarrow X$ be a  superconvex pseudograph. 
There is $M'>0$ such that for any local-isometry $W\rightarrow Y$, any wall-piece of $W\rightarrow X$ has diameter at most $M'$. Hence, $W \rightarrow X$ is itself superconvex.
\end{lem}

\begin{proof}
For $W\rightarrow Y$, any wall-piece of $W\rightarrow X$  maps  to either a hyperplane of $\widetilde Y$ or to a wall-piece of $Y\rightarrow X$. Both of these have uniformly bounded diameter.

Indeed, in the first case, the wall-piece is uniformly bounded because $Y$ is a pseudograph (in particular, compact) and therefore has finitely many hyperplanes, all of which are contractible. In the second case,  the wall-piece  is  uniformly bounded  because $Y$ is compact and superconvex.
\end{proof}

%%%%%%%%%%%%%%%%%%%%%%%%%%%%%%%%%%%%%%%%%%%%%%%%%%%%%%%%%%%%%%%%%%%%%
\subsection{Passing to a malnormal free subgroup}\label{sec: passing to malnormal}
%%%%%%%%%%%%%%%%%%%%%%%%%%%%%%%%%%%%%%%%%%%%%%%%%%%%%%%%%%%%%%%%%%%%%

\begin{defn}[fiber-product]
Let $f:Y\rightarrow X$ be a map.
The \emph{fiber-product} $Y \otimes_X Y$ is the subspace of $Y\times Y$
consisting of the preimage of the diagonal $\{(x,x): x \in X\}$ using the map
$f\times f : Y\times Y \rightarrow X\times X$.
The \emph{diagonal} component of $Y\otimes_X Y$ is the subspace $\{(y,y): y\in Y\}$.
The projection maps $Y \leftarrow Y\times Y\rightarrow Y$ induce projection maps
 $Y \leftarrow Y\otimes_X Y\rightarrow Y$.
 
When $f$ is a local-isometry  map between nonpositively curved cube complexes,  $Y\otimes_X Y$ is a  nonpositively curved cube complex, and the projection maps are local-isometries.
Concretely, the $i$-cubes of $Y \otimes_X Y$ are pairs of $i$-cubes in $Y$ that map to the same $i$-cube of $X$.

\end{defn}

\begin{lem}\label{no cover}
    No  non-diagonal component of $Y\otimes_X Y$ is a finite cover of $Y$.
\end{lem}
\begin{proof}
Suppose $\widehat Y$ is a non-diagonal component of $Y\otimes_X Y$
that is a finite cover of $Y$.
Let  $x$ the basepoint of $X$, and let $y$ be a point of $Y$ mapping to $x$.
Let $(y,y')$ be a point of $\widehat Y$ that left-projects to $y$.
Since $(y,y')$ is not in the diagonal, we see that $y'\neq y$.
However, both $y$ and $y'$ project to $x$.
Let $\kappa$ be a path in $Y$ from $y$ to $y'$.
Then $[\kappa]\in \pi_1X$ stabilizes $\widetilde Y$,
but $[\kappa]\notin \pi_1Y$ since its based lift to $Y$ is not closed.
Thus $F\subsetneq \bar F= \langle F, [\kappa]\rangle$.
Let $\bar Y=\bar F \backslash \widetilde Y$.
Then there is a proper covering map $Y\rightarrow \bar Y$.
Let $d$ be the degree of this proper covering map, and note $1<d$ by properness
and $d<\infty$ since $Y$ is compact.
However, $-1 = \euler(Y)= d\cdot\euler(\bar Y) \neq -1$, which is impossible.
\end{proof}

\begin{prop}\label{prop: no conj}
    Let $Y\to X$ be a superconvex rank~$2$ pseudograph.
    There exists a local isometry $W\to Y$ where $W$ is a superconvex rank~$2$ pseudograph, such that no nontrivial element of $\pi_1W$ is conjugate to an element in $\pi_1$ of a non-diagonal component of $Y\otimes_X Y$. 
    
    Moreover, we can assume that $\pi_1W$ is malnormal in $\pi_1 Y$.
\end{prop}
\begin{proof}
Let $K_1,\ldots, K_m$ be the fundamental groups of the non-diagonal components of $Y \otimes_X Y$.
Each $K_i$ is finitely generated by compactness.
By Lemma~\ref{no cover}, no non-diagonal component of $Y \otimes_X Y$ is a finite cover of $Y$.
Thus each $K_i$ is an infinite index subgroup of the free group $\pi_1 Y$. Let $H\subset \pi_1 Y$ be the 
subgroup provided %below
by Lemma~\ref{lem:nonconjugate}.
Let $W\rightarrow Y$ be a local-isometry with $W$ compact such that $\pi_1W=H$. Then $W\rightarrow Y$ satisfies the statement of the lemma.
\end{proof}

Above we used the following statement about %finitely generated
subgroups of free groups. This was also noted in \cite[Thm~D]{Kapovich99}.

\begin{lem}[A free group lemma]\label{lem:nonconjugate}
    Let $F$ be a free group on at least two generators, and let $K_1, \dots, K_n$ be finitely generated infinite index subgroups. There exists a subgroup $H\subseteq F$ isomorphic to $F_2$ such that no nontrivial element of $H$ is conjugate to an element of $\cup_i K_i$.

    Moreover, we can assume $H$ is malnormal in $F$.
    \end{lem}

We will use the graphical viewpoint on subgroups of free groups
popularized by Stallings \cite{Stallings83}.

\begin{proof}
Regard $F$ as $\pi_1B$ where $B$ is a bouquet of circles.
For each $K_i$ let $B_i\rightarrow B$ be an immersion of a finite based graph
with $\pi_1B_i$ mapping to $K_i$.

We will produce two immersed cycles $\sigma_1\rightarrow B$
and $\sigma_2\rightarrow B$,
such that $\sigma_1\vee \sigma_2$ maps to $B$ by an immersion,
and such that the based path $\sigma_j\rightarrow B$ does not have a lift to $B_i$ for any $i,j$.
The result follows letting $H=\sigma_1\vee \sigma_2$.

It suffices to produce one such path $\sigma$,
since we can precompose and postcompose it with edges to produce another immersed cycle with the same property, and such that their wedge immerses in $B$.

Let $C=\cup_i B_i$.
Let $\{v_1,\ldots, v_s\}$ be the vertices of $C$.
Let $C^+\rightarrow B$ be the covering map obtained by adding finitely many trees to $C$.
A path $\mu$ \emph{exits} $C$, if $\mu$ traverses an edge of one of these trees.

We produce $\sigma$ through the following recursion:
Let $\mu_0\rightarrow B$ be the trivial path.
For each $0\leq m \leq s-1$,
consider the lift of $\mu_m$ at $v_m$, and let $\lambda_{m+1}$
be a (possibly trivial) path such that $\mu_{m+1}:=\mu_m\lambda_{m+1}$ exits $C$.
The path $\sigma=\mu_s$ has the desired property by construction.

For the malnormality, suppose $U,V$ are closed paths in $B$
such that $U\vee V$ immerses in $B$.
Let $U'=UVUV^2UV^3\cdots U V^nU$
and $V'=VUVU^2\cdots VU^n V $.
Then $U'\vee V'$ immerses in $B$,
and $U'\vee V' \rightarrow B$ is small-cancellation and hence malnormal
for $n\gg 0$. See~\cite[Thm~2.14]{WisePositive}.
\end{proof}

%%%%%%%%%%%%%%%%%%%%%%%%%%%%%%%%%%%%%%%%%%%%%%%%%%%%%%%%%%%%%%%%%%%%%
\subsection{Induced small-cancellation conditions}\label{sec: induced small-cancellation}
%%%%%%%%%%%%%%%%%%%%%%%%%%%%%%%%%%%%%%%%%%%%%%%%%%%%%%%%%%%%%%%%%%%%%

\begin{prop}\label{prop:small-cancellation_correspondence}
Let $W\to X$ be a superconvex pseudograph such that the non-diagonal components of $W \otimes_X W$ are contractible (equivalently, $\pi_1W$ is malnormal in $\pi_1X$).

Let $B$ be a graph and  $B\rightarrow W$ be a combinatorial homotopy equivalence.

There exist $\kappa, \epsilon>0$, and $\beta = \beta(\kappa, \epsilon)>0$, such that for any immersion
$A\rightarrow B$, there exists a local isometry $Z\rightarrow W $ such that 
\begin{itemize}
    \item $\sys(Z)\geq \frac 1 \kappa \sys(A)-\epsilon$,
    \item if $\langle B \mid A \rangle$ is $C'(\alpha)$ then
$\langle X\mid Z\rangle$ is $C'(\beta\alpha)$.
\end{itemize}
\end{prop}

We use the notation $\neb_r(S)$ for the closed $r$-neighborhood of $S$.

\begin{proof}
The map $B\rightarrow W$ lifts to a $(\kappa,\epsilon)$-quasiisometry $\phi: \widetilde B\rightarrow \widetilde W$
for some $\kappa,\epsilon>0$.
The correspondence between (compact) immersions $A\rightarrow B$ and (compact) immersions $Z\rightarrow W$
arises by letting  $Z$ be the quotient of $\widetilde Z$ by the action of $\pi_1A$, where $\widetilde Z $ is the convex hull of $\phi(\widetilde A)$.
We note that there is a uniform constant~$r$ with $\widetilde Z\subset \neb_r(\phi(\widetilde A))$.
We refer to  \cite[Thm~2.28]{HaglundGraphProduct} and \cite[Prop~3.3]{SageevWiseCores},
and note that the constant $\mu$ in the proof of the second reference is uniform since $\phi(\widetilde A)$ is uniformly quasiconvex. 

There is a uniform linear relationship between the systole of $Z$ and the systole of $A$, namely $\frac {1}{\kappa} \sys(A)- \epsilon \leq \sys (Z)\leq \kappa\sys(A)+\epsilon$.
By Lemma~\ref{lem:bounded wall pieces} there is a uniform bound on diameters of wall-pieces of $Z\rightarrow X$.

A cone-piece $U$ in $Z\rightarrow X$
corresponds to a pair of maps $Z\leftarrow U \rightarrow Z$
which composes to a pair of maps $W\leftarrow Z \leftarrow U \rightarrow Z\rightarrow W$.
The universal property of the fiber-product provides a lift of $U$ to $W \otimes_X W$.
The non-diagonal components of $W\otimes_X W $ have
uniformly bounded diameter, since they are compact and contractible by assumption. Thus there is a uniform bound on the diameter of pieces $U$ mapping to the non-diagonal components of $W\otimes_X W $.

Suppose $U$ maps to the diagonal component of $W\otimes_X W$ and so $U$ lifts to a component of $Z\otimes_W Z$.
Then $U$ is a piece in $Z\rightarrow W$, i.e.\ a component of the intersection between translates $g\widetilde Z$ and $\widetilde Z$ in $\widetilde W$. 
There exists a corresponding piece $V$ in $A\to B$, which is the intersection between $g\widetilde A$ and $\widetilde A$ in $\widetilde B$. We now claim that $\diam(V)\geq \kappa'\diam(U) - \epsilon'$ for some $\kappa', \epsilon'>0$.

For $a_1,a_2\in \widetilde A$, if $\dist(\phi(a_1),\phi(a_2)) <r$
then $\dist(a_1,a_2)< \kappa r + \epsilon$.
Suppose $\diam(\widetilde Z\cap g\widetilde Z)=M$,
then there are points $w_1,w_2$ inside with $\dist(w_1,w_2)=M$.
Let $a_i$ be a point in $\widetilde A$
with $\dist(\phi(a_i),w_i) < r$, and let $a_i'$ be a point in $g\widetilde A$ such that $\dist(\phi(a_i'), w_i)<r$. In particular, $\dist(a_i, a_i')\leq \kappa(2r)+\epsilon \kappa$. 
On the other hand $\dist(\phi(a_1),\phi(a_2))\geq M-2r$, so $\dist(a_1,a_2)\geq \frac{1}{\kappa}(M-2r)-\frac{\epsilon}{\kappa}$. 

Note that $a_1, a_2$ belong to the convex subset $\widetilde A$ of the tree $\widetilde B$, and so the unique path $[a_1, a_2]\subseteq \widetilde A$. Similarly $[a_1', a_2']\subseteq g\widetilde A$. Since points $a_i,a_i'$ are close for both $i=1,2$, and points $a_1,a_2$ are far away, it follows that there exist points $b_i\in [a_i, a_i']$ with $[b_1, b_2]\subseteq \widetilde A\cap g\widetilde A$. 
We have $\dist(b_i, a_i)\leq \kappa(2r)+\epsilon \kappa$, and therefore $$\dist(b_1, b_2) \geq \dist(a_1, a_2) - \dist(b_1, a_1)- \dist(b_2, a_2) \geq \frac{1}{\kappa}(M-2r)-\frac{\epsilon}{\kappa} - 2 (\kappa(2r)+\epsilon \kappa).$$
Thus $\diam(V)\geq \kappa'\diam(U) - \epsilon'$ for $\kappa' = \frac 1  \kappa$, and $\epsilon' = \frac{2r+\epsilon}{\kappa} + 2\kappa(2r+\epsilon)$.

To summarize, every piece $U$ in $W$ either has diameter bounded by some universal constant, or satisfies $$\diam(U)\leq \frac{1}{\kappa'} \diam(V) +\frac{\epsilon'}{\kappa'} \leq \frac{1}{\kappa'} \alpha\sys(A) +\frac{\epsilon'}{K'} \leq \frac{1}{\kappa'} (\alpha \kappa\sys(W)+\epsilon) +\frac{\epsilon'}{\kappa'}.$$ In either case, we get that $\diam(U)\leq \beta\alpha \sys(W)$ for some $\beta>0$ which depends on $W\to X$ and $B\to W$ only.
\end{proof}

\begin{rem}\label{rem: small cancellation in graphs}
    For each $\alpha>0$, there exists  $A\to B$ where $\langle B \mid A \rangle$ is $C'(\alpha)$ and $\rank(A)=2$. For instance, if $B$ is  a bouquet of circles labelled by $a$ and $b$,  choose $A$ associated to $\langle aba^2b \cdots a^mb, bab^2a\cdots b^na \rangle$ for  sufficiently large $m,n$.
\end{rem}

Proposition~\ref{prop:small-cancellation_correspondence} and Remark~\ref{rem: small cancellation in graphs} imply the following.
\begin{cor}\label{cor: small cancellation} Let $W\to X$ be a superconvex pseudograph where the non-diagonal components of $W \otimes_X W$ are contractible.
For any $\alpha>0$ there is a local isometry of superconvex rank~$2$ pseudograph $Z\to W \to X$ such that $\langle X\mid Z\rangle$ is $C'(\alpha)$. 
Moreover, for any $R>0$ we can choose $Z$ so that $\sys(Z)\geq R$.
\end{cor}

To finish this section we note that Proposition~\ref{prop:small-cancellation_correspondence} does not assume that $A$ and $Z$ are connected. Thus we can thus deduce the following improvement, by choosing appropriate $A_1, A_2\to B$ as in Remark~\ref{rem: small cancellation in graphs} whose cancellations with one another are small.

\begin{cor}\label{cor:2 pseudographs}
    Let $W\to X$ be a superconvex pseudograph such that the non-diagonal components of $W \otimes_X W$ are contractible.
    Then for every $\alpha, R>0$ there exists a pair of superconvex rank~$2$ pseudographs $Z_1\to W\to X$ and $Z_2\to W \to X$ with $\sys(Z_i)\geq R$ and $\langle X\ \mid Z_1, Z_2\rangle$ satisfying $C'(\alpha)$.
\end{cor}

%%%%%%%%%%%%%%%%%%%%%%%%%%%%%%%%%%%%%%%%%%%%%%%%%%%%%%%%%%%%%%%%%%%%%
\subsection{Proof of  a proper $C'(\alpha)$ small-cancellation quotient}\label{sec: proof without cocompactness}
%%%%%%%%%%%%%%%%%%%%%%%%%%%%%%%%%%%%%%%%%%%%%%%%%%%%%%%%%%%%%%%%%%%%%
We are now ready to prove the first part of our main theorem.
\begin{thm}[Theorem~\ref{thm:main} without cubulation]\label{thm: main without cubulation}
     Let $X$ be a nonpositively curved cube complex that admits a local isometry of a rank~$2$ superconvex pseudograph. 
    Then for every $\alpha> 0$ there exists $Y\looparrowright X$ with $\pi_1 Y\neq 1$ such that $\langle X\mid Y\rangle$ is a cubical $C'(\alpha)$ small-cancellation presentation. 
\end{thm}

\begin{proof}
Proposition~\ref{prop: no conj} provides a local isometry $W\rightarrow Y$ so that all non-contractible components of  $W \otimes_X W$ map to the diagonal component of $Y\otimes_X Y$.
Since we can choose $W$ such that $\pi_1 W$ is malnormal in $\pi_1 Y$, we can assume that all non-diagonal components of $W \otimes_X W$ are contractible.
By Corollary~\ref{cor: small cancellation} there is $Z\to W$ such that $\langle X\mid Z\rangle$ is a $C'(\alpha)$ small-cancellation quotient.
\end{proof}

 Recall that  by possibly passing to the $\pi_1 X$-invariant subcomplex of the cubical subdivision of $\widetilde X$, we can assume  $\pi_1 X$ acts essentially on $\widetilde X$ by \cite[Prop~3.5]{CapraceSageev2011}. Then  a rank~$2$ superconvex pseudograph $Y\rightarrow X$
 exists by Proposition~\ref{prop:the core}.
 
 Similarly as in Theorem~\ref{thm: main without cubulation}, we deduce a version of $SQ$-universality for the fundamental groups of non-product nonpositively curved cube complexes.

\begin{prop}[Subgroup Quotient Universal]\label{prop:squniversal}Fix $\alpha > 0$. There exists $R = R(\alpha)$ such that the following holds.
Suppose $X$ admits a local isometry from a rank~$2$ pseudograph $Z\rightarrow X$  such that
\begin{itemize}
    \item  wall pieces are uniformly bounded,
    \item non-diagonal components of $Z\otimes_X Z$ are contractible,
    \item $\sys(Z)\geq R$.
\end{itemize}
Then for every rank~$2$ group $H$,
there is a $C'(\alpha)$ cubical small-cancellation quotient
$X^*=\langle X \mid \widehat Z\rangle$ with
$H\hookrightarrow \pi_1X^*$ where $\widehat Z \to Z$ is a covering map. 
\end{prop}

Compactness is not assumed for $\widehat Z$.
The existence of $Z$ is ensured by Lemma~\ref{lem:bounded wall pieces}, 
Proposition~\ref{prop: no conj}, and Corollary~\ref{cor: small cancellation}.
The same proof works for rank~$m$, but anyhow, every countable group is a subgroup of a quotient of a rank~$2$ free group.
\begin{proof}
Choose a regular cover $\widehat Z\rightarrow Z$ 
with  $H\cong \Aut(\widehat Z \to Z)$. 
Every cone-piece of $\widehat Z \rightarrow X$ is a cone-piece of $Z\rightarrow X$. But $\sys(\widehat Z)\geq \sys(Z)$.
Finally, since $\widehat Z \subset \widetilde X^*$, then $H\cong \Aut(\widehat Z \rightarrow Z)\subset  \Aut(\widetilde X^*\rightarrow X) =\pi_1X^*$. Indeed  every automorphism of $\widehat Z \rightarrow Z$  extends to an automorphism of $\widetilde X^*\rightarrow X$.
\end{proof}

%%%%%%%%%%%%%%%%%%%%%%%%%%%%%%%%%%%%%%%%%%%%%%%%%%%%%%%%%%%%%%%%%%%%%
%%%%%%%%%%%%%%%%%%%%%%%%%%%%%%%%%%%%%%%%%%%%%%%%%%%%%%%%%%%%%%%%%%%%%

%%%%%%%%%%%%%%%%%%%%%%%%%%%%%%%%%%%%%%%%%%%%%%%%%%%%%%%%%%%%%%%%%%%%%
\section{Disc diagrams and more  cubical small-cancellation }\label{sec: disc diagrams}
%%%%%%%%%%%%%%%%%%%%%%%%%%%%%%%%%%%%%%%%%%%%%%%%%%%%%%%%%%%%%%%%%%%%%

A \emph{disc diagram} $D$ is a compact contractible combinatorial 2-complex, together with an embedding $D \hookrightarrow S^2$. The \emph{boundary path} $\partial D$ is the attaching map of the 2-cell at infinity.

Similarly, an \emph{annular diagram} $A$ is a compact combinatorial 2-complex homotopy equivalent to $S^1$, together with an embedding $A \hookrightarrow S^2$, which induces a cellular structure on $S^2$. The \emph{boundary paths} $\partial_{int}A$ and $\partial_{out}A$ of $A$ are the attaching maps of the two 2-cells in this cellulation of $S^2$ that do not correspond to cells of $A$.

A \emph{disc (resp.\ annular) diagram in $X^*$} is a combinatorial map $(D,\partial D)\to(X^*, X^1)$ of a disc diagram (resp. $(A, \partial A) \rightarrow (X^*, X^1)$). 
A \emph{square disc (resp. annular) diagram} is a disc (resp. annular) whose image is entirely contained in $X$ (without cones).

The $2$-cells of a disc diagram $D$ in $X^*$ are of two kinds: squares mapping onto squares of $X$, and triangles mapping onto cones over edges contained in $Y_i$. 
The vertices in $D$ which are mapped to the cone-vertices of $X^*$ are also called the \emph{cone-vertices}. 
Triangles in $D$ are grouped into cyclic families meeting around a cone-vertex. 
We refer to such families as \emph{cone-cells}, and treat a whole such family as a single $2$-cell. 

Let $D$ be a disc diagram in $X^*$. The \emph{square part} $D_\square$ of $D$ is the union of all the squares in $D$
that are not contained in any cone-cells. 

The \emph{complexity} of a disc diagram $D$ in $X^*$ is defined as $$\Comp(D) = (\# \text{cone-cells in } D, \#\text{squares in $D_\square$}).$$ We say that $D$  has \emph{minimal complexity} if $\Comp(D)$ is minimal in the lexicographical order among disc diagrams with the same boundary path as $D$. A disc diagram $D$ in $X^*$ is \emph{degenerate} if $\Comp(D) = (0,0)$.
A disc diagram $D,$ in $X^*$ is \emph{singular} if $D$ is not homeomorphic to a closed ball in $\mathbb R^2$. This is equivalent to $D$ either being a single vertex or an edge, or containing a cut vertex. 
In particular, every degenerate disc diagram is singular.

\begin{definition}[$n$-collared]\label{defn: collared}
    A disc diagram $D \rightarrow X^*$ is \emph{collared} by an annular diagram $A \rightarrow D\to X^*$ if there is a subdiagram $D_0\subseteq D$ such that $$D= A \bigsqcup\limits_{\partial_{in} A = \partial D_0} D_0.$$
    
Note that if $\partial_{in}A$ does not embed in $D$, then $D_0$ is singular. Also $\partial D = \partial_{out}A$.

If $A$ decomposes as a union of $n$ ladders $L_1, \ldots, L_n$, and each cone-cell of each $L_i$ intersects  $\partial_{out} A$ and $\partial_{in} A$ nontrivially, then $D$ is \emph{n-collared} by $A$. We refer to $A$ as the \emph{collar} of $D$, to $L_1, \ldots, L_n$ as \emph{collaring ladders}.
\end{definition}

\subsection{Greendlinger's Lemma}
A cone-cell $C$ in a disc diagram $D$ is a \emph{boundary cone-cell} if $C$ intersects the boundary $\partial D$ along at least one edge. A non-disconnecting boundary cone-cell $C$ is a \emph{shell of degree $k$} if  $\partial C = RQ$ where 
$Q$ is the maximal subpath of $\partial C$ contained in $\partial D$, and $k$ is the minimal number such that $R$ can be expressed as a concatenation of $k$ pieces. We refer to $R$ as the \emph{innerpath} of $C$ and $Q$ as the \emph{outerpath} of $C$. A shell  of degree $\leq 4$ is an \emph{exposed shell}.

A \emph{corner of a square} in a disc diagram $D$ is a vertex $v$ in $\partial D$ of valence $2$ in $D$ that is contained in some square of $D$. 
A \emph{cornsquare} is a square $c$ and  a pair of dual curves emanating from consecutive edges $a, b$ of $c$ that terminate on consecutive edges $a',b'$ of $\partial D$. We abuse  notation and refer to the common vertex of $a',b'$ as a cornsquare as well.
A \emph{spur} is a vertex in $\partial D$ of valence $1$ in $D$. If $D$ contains a spur or a cut-vertex, then $D$ is \emph{singular}.

A \emph{pseudo-grid} between paths $\mu$ and $\nu$ is a square disc diagram $E$ where the boundary path $\partial E$ is a concatenation $\mu \rho \overline{\nu} \overline{\eta}$ where each $\mu, \rho, \overline{\nu}, \overline{\eta}$ is a called a \emph{side}, and  such that 
\begin{enumerate}
\item each dual curve starting on $\mu$ ends on $\nu$, and vice versa, 
\item no pair of dual curves starting on $\mu$ cross each other,
\item no pair of dual curves cross each other twice.
\end{enumerate}
If a pseudo-grid $E$ is degenerate then either $\mu = \nu$ or $\rho= \eta$. A \emph{grid} is a pseudogrid isometric to
the product of two intervals.

A \emph{ladder} is a 
disc diagram $(D, \partial D)\to (X^*, X^0)$ which is an alternating union of cone-cells and/or vertices $C_0, C_2\dots, C_{2n}$ and (possibly degenerate) pseudo-grids $E_1, E_3\dots, E_{2n-1}$, with $n\geq 0$, in the following sense:
\begin{enumerate}
\item the boundary path $\partial D$ is a concatenation $\lambda_1\overline{
\lambda_2}$ where the initial points of $\lambda_1, \lambda_2$ lie in $C_0$, and the terminal points of $\lambda_1,\lambda_2$ lie in $C_{2n}$,
\item $\lambda_1 =\alpha_0 \rho_1\alpha_2\cdots \alpha_{2n-2}\rho_{2n-1}\alpha_{2n}$ and $\lambda_2 =\beta_0\eta_1\beta_2\cdots \beta_{2n-2}\eta_{2n-1}\beta_{2n}$,
\item the boundary path $\partial C_{i} = \nu_{i-1}\alpha_i \overline{\mu_{i+1}}\overline{\beta_i}$ for some $\nu_{i-1}$ and $\mu_{i+1}$ (where $\nu_{-1}$ and $\mu_{2n+1}$ are trivial), and
\item the boundary path $\partial E_{i} = \mu_i \rho_i \overline{\nu_i} \overline{\eta_i}$.
\end{enumerate}

The cubical version of Greendlinger's Lemma will be  used in Section~\ref{sec: wall traingles}. See~\cite[Thm~3.46]{WiseBook} and~\cite[Thm~2]{Jankiewicz2017} for the proof.

\begin{lem}[Cubical Greendlinger's Lemma]\label{lem:Greendlinger}
Let $X^*= \langle X\mid Y_1, \dots, Y_s\rangle$ be a cubical presentation satisfying the $C(9)$ condition. Let $D\to X^*$ be a minimal complexity disc diagram. Then one of the following holds:
\begin{itemize}
\item $D$ is a ladder, or
\item $D$ has at least three exposed shells and/or corners of squares and/or spurs.
\end{itemize}
\end{lem}

\begin{theorem}[The Ladder Theorem] \label{thm:ladder theorem}
Let $X^*= \langle X\mid Y_1, \dots, Y_s\rangle$ be a cubical presentation satisfying the $C(9)$ condition and let $D\to X^*$
be a minimal complexity disc diagram in
$X^*$. If $D$ has exactly two exposed shells, then $D$ is a ladder.
\end{theorem}

\subsection{The $B(6)$ condition}

We now introduce a  set of conditions that provides
a wallspace structure on $\widetilde X^*$ so that $\pi_1X$ acts on a CAT(0) cube complex:
\begin{definition}\label{def:b6}  A cubical presentation $\langle X \mid \{Y_i\} \rangle $ satisfies
the \emph{B(6) condition} if the following conditions are satisfied:

\begin{enumerate}
\item \label{item:b6.1}(Small-cancellation) $\langle X \mid \{Y_i\} \rangle$ satisfies the $C'(\frac{1}{\alpha})$ condition for $\alpha \geq 14$.
\item \label{item:b6.2} (Wallspace Cones) Each $Y_i$ is a wallspace where each wall in $Y_i$ is the union $\sqcup H_j$ of a collection of disjoint embedded 2-sided hyperplanes in $Y_i$, and there is an embedding $\sqcup N(H_j) \rightarrow Y_i$ of the disjoint union of their carriers into $Y_i$. Each such collection separates $Y_i$. Each hyperplane in $Y_i$ lies in a unique wall.
\item \label{item:b6.3} (Hyperplane Convexity) If $P \rightarrow Y_i$ is a path that starts and ends on vertices
on 1-cells dual to a hyperplane $H$ of $Y_i$ and $P$ is the concatenation of at most 7 pieces, then $P$ is path homotopic
in $Y_i$ to a path $P\rightarrow N(H) \rightarrow Y_i$.
\item \label{item:b6.4} (Wall Convexity) Let $S$ be a path in $Y_i$ that starts and ends with 1-cells
dual to the same wall of $Y_i$. If $S$ is the concatenation of at most 7 pieces, then $S$ is path-homotopic into the carrier of a hyperplane of that wall.
\item \label{item:b6.5} (Equivariance) The wallspace structure on each cone over $Y_i$ is preserved by $\Aut_X(Y)$. 
\end{enumerate}
\end{definition}

\begin{definition}[$k$-Wall Convexity]
    We will consider a variant of Definition~\ref{def:b6}.\eqref{item:b6.4} where $7$ is replaced with $k$. We call it \emph{$k$-Wall Convexity}.
\end{definition}
This condition will be used in the proof of cocompactness of the action of $\pi_1X^*$.

\section{Cubulating}\label{sec: cubulating}

The purpose of this section is to set up the proof of the remaining part of Theorem~\ref{thm:main}, i.e.\ that $Y$ can be chosen so that $\pi_1 X^*$ is  cubulated, and prove some additional properties of the cubulation (Theorem~\ref{thm:main with all properties}). In Section~\ref{subsec: wallspace structure} we describe how to obtain the cubical structure and some of its properties. In Section~\ref{subsec:induced}
we discuss induced cubical presentations, which are used in some parts of the proof of Theorem~\ref{thm:main with all properties}.
In Section~\ref{sec:freeness}, we prove that $Y$ can be chosen so that $\pi_1 X^*$ acts freely on its dual.
Section~\ref{sec:cocompact} contains a result that implies cocompactness of the action on the dual under appropriate assumptions on $X$.

\subsection{Wallspace structures on pseudographs}\label{subsec: wallspace structure}

Let $X^*=\langle X \mid \{Y_i\}\rangle$ be a cubical presentation where each $Y_i$ is a rank~$1$ pseudograph. There is a standard wallspace structure that can be put on each $Y_i$ to obtain a wallspace structure on $\widetilde X^*$, which we now describe. 

\begin{construction}[Wallspace structure on rank~$1$ pseudographs]\label{const: rank 1 wall}
Let $Y \rightarrow X$ be a rank~$1$ pseudograph and let $\sigma \rightarrow Y$ be a closed-geodesic in $Y$  realising the systole, so $|\sigma|=sys(Y)$. After potentially barycentrically subdividing $X$, we may assume that $|\sigma|=2n$ for some $n$, so that there is a well-defined  equivalence relation on the hyperplanes of $Y$ where two hyperplanes that are dual to edges of $\sigma$ are equivalent if and only if they are dual to antipodal edges of $\sigma$.  If a hyperplane is not dual to an edge of  $\sigma$, it is its own equivalence class. This equivalence relation then defines a wallspace structure on $Y$.
\end{construction}

\begin{lemma}\label{lem: cyclic wallspace b6}
    Let $X^*=\langle X \mid \{Y_i\}\rangle$ be a cubical presentation satisfying the $C'(\frac{1}{n})$ condition for $n \geq 16 $ and where each $Y_i$ is a rank~$1$ pseudograph. Then $X^*$ satisfies the  $B(6)$ condition with each $Y_i$ endowed with the wallspace structure in Construction~\ref{const: rank 1 wall}. Thus $\pi_1 X^*$ acts on the CAT(0) cube complex dual to the wallspace structure on $X^*$ obtained by extending the wallspace structures on the $Y_i$.
\end{lemma}

 Versions of Lemma~\ref{lem: cyclic wallspace b6} are proven in~\cite{ArenasAGT, FuterWise, WiseBook}.

With a bit of care, the ideas in Construction~\ref{const: rank 1 wall} can be extended to work for higher-rank pseudographs; we explain this in detail for the rank~$2$ case below.

\begin{construction}[Wallspace structure on rank~$2$ pseudographs]\label{const: walspace rank 2}
Let $Y$ be a rank~$2$ superconvex pseudograph
and let $\pi_1Y = \langle a,b\rangle$.
Consider rank~$1$ pseudographs $Z_{1}$ and $Z_{2}$
corresponding to the subgroups $\langle aba^2b\cdots a^mb \rangle$ and $\langle  bab^2a\cdots b^n\rangle$ for some appropriately large $n,m$.

There exists $r=r(Y,a,b)>0$ such that
the local isometry $Y'\rightarrow X$ 
associated to  
$\langle aba^2b\cdots a^mb , bab^2a\cdots b^n\rangle$
has the following form:
There is a contractible locally convex subcomplex $K\subset Y'$ with
\begin{enumerate}
\item 
 $\diam(K)\leq r$ such that:
    
    \item $Y'$ is the union  $Z_1\cup K \cup Z_2$.
    \item $Z_1\cap Z_2 \subset K$.
\end{enumerate}

As in the rank~$1$ case in Construction~\ref{const: rank 1 wall}, we may subdivide $X$ and consider the antipodal wallspace structure on each of $Z_1$ and $Z_2$ with respect to some choice of $\sigma_i\to Z_i$ where $|\sigma_i| = \sys{Z_i}$.
We extend this to a wallspace structure on $Y'$ as follows:
\begin{enumerate}
    \item Each wall of $Z_1$ disjoint from $K$ is a wall of $Y'$.
    \item Each wall of $Z_2$ disjoint from $K$ is a wall of $Y'$.
    \item Each hyperplane $U$ of $K$  extends to a wall of $Y'$ consisting of:
    $U$, the  hyperplane of $Z_1$ that is antipodal to $Z_1\cap U$ (if non-empty),
    and
    the  hyperplane of $Z_2$ that is antipodal to $Z_2\cap U$ (if non-empty).
    \end{enumerate}
See Figure~\ref{fig:Pseudograph_Wallspace}.
\end{construction}

\begin{figure}
\centering
\includegraphics[height=1.3in]{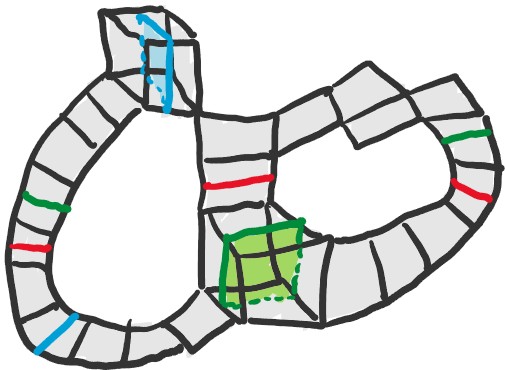}
\caption{Some walls in the wallspace of Construction~\ref{const: walspace rank 2}.
\label{fig:Pseudograph_Wallspace}}
\end{figure}

We prove in Theorem~\ref{thm: main without cubulation} that for each $\alpha>0$ and  for all 
for $n,m\gg 1$,
$\langle X \mid Y'\rangle$ is $C'(\alpha)$,  in Theorem~\ref{thm:main with all properties}.\eqref{conclusions:B6} that  $\langle X \mid Y'\rangle$ satisfies  the  $B(6)$ condition, and in Theorem~\ref{thm:main with all properties}.\eqref{conclusions:k-wall convex} that $\langle X \mid Y'\rangle$ satisfies $25$-wall convexity. In Lemma~\ref{lem: wallspace structure freeness} that  $\langle X \mid Y'\rangle$ satisfies the hypotheses of Theorem~\ref{thm:proper} with this wallspace structure on $Y'$. Finally we prove in~\ref{thm:freeness} that 
 $\langle X \mid Y'\rangle$ has torsion-free fundamental group.

\begin{rem}
    The choice of 
    $\langle aba^2b\cdots a^mb \rangle$ and $\langle  bab^2a\cdots b^n\rangle$ is rather arbitrary, and the construction and proof work quite generally for a pair of small-cancellation words.
\end{rem}

\begin{rem}\label{rem:diam(K)}
    Note that in Construction~\ref{const: walspace rank 2} for any $C$ we can choose $Y'$ such that $\sys(Y')\geq C\diam(K)$. Indeed, $\diam(K)$ is proportional to the overlap between the words defining $Z_1, Z_2$.
\end{rem}

\begin{rem}\label{rem: trivial aut}
    Note that $\Aut_X(Y')$ is trivial in Construction~\ref{const: walspace rank 2}. Indeed, as $Y$ is compact, $|\Aut_X(Y')|< \infty$. Thus, $[\Stab_X(\widetilde Y'):\pi_1Y']< \infty$ and if $\Stab_X(\widetilde Y')\neq \pi_1Y'$, then there is a local isometry $Y_0 \rightarrow X$ such that $Y'$ is a proper, regular, finite-degree covering of  $Y_0$.  This contradicts the choice of $Y'$.
\end{rem}

%%%%%%%%%%%%%%%%%%%%%%%%%%%%%%%%%%%%%%%%%%%%%%%
\subsection{Induced Cubical Presentations}\label{subsec:induced}
%%%%%%%%%%%%%%%%%%%%%%%%%%%%%%%%%%%%%%%%%%%%%%%

\begin{defn}[Induced cubical presentation]
\label{def:induced presentation}
Given a cubical presentation $X^*=\langle X \mid \{Y_i\}\rangle$ and a local-isometry $f:E\rightarrow X$ there is an \emph{induced cubical presentation}
$E^*=\langle E \mid \{E\otimes_X Y_i\} \rangle$. 
For each $i \in I$ there is an induced map $f_i: E\otimes_X Y_i \rightarrow Y_i$ such that the following diagram commutes.

\begin{center}
\begin{tikzcd}
E\otimes_X Y_i \arrow{r}{f_i} \arrow{d} & Y_i \arrow{d} \\
E \arrow{r}{f}                        & X            
\end{tikzcd}
\end{center}

A \emph{map of cubical presentations} is the data $(X^*, E^*, f)$. Note that $f$ induces a combinatorial map $f^*:E^* \rightarrow X^*$ that sends  cone-points to cone-points and the cubical part of $E^*$ to the cubical part of $X^*$.

\end{defn}

\begin{definition}[Liftable shells in induced presentations]\label{def:nomissingshells}
    Let $X^*$ be a cubical  presentation. Let $f:E\to X$ be a local isometry. We say that the induced presentation $f^*:E^* \rightarrow X^*$ has \emph{liftable shells} if the following holds:
    
    Let $R$ be a non-replaceable (i.e.\ whose boundary path is not null-homotopic within its associated cone over $Y_j$)
    shell in a minimal complexity diagram $D\rightarrow X^*$ with $\partial R = QS$, where $Q$ is its outerpath and $S$ is its innerpath.
    If $Q$ lifts to $E$, then so does $S$.
\end{definition}

The above condition is relevant due to the following theorem.

\begin{thm}[{\cite[Thm~3.68]{WiseBook}}]\label{thm:no_missing_shells_injective}
     Let $f: E^* \rightarrow X^*$ be a map of cubical presentations, where $X^*$ satisfies $C'(\frac 1 {12})$. If $f$ has liftable shells then $f$ is $\pi_1$-injective, and $\widetilde E^*\to \widetilde X^*$ is injective on the cubical part.
\end{thm}

The liftable shells property in induced presentations is guaranteed by:
\begin{lem}[{\cite[Lem~3.67]{WiseBook}}]\label{lem:liftable shells condition}
Let $ \langle X \mid \{Y_i\} \rangle$  be a $C'(\frac{1}{14})$ cubical presentation. Let $A \rightarrow X$ be a local-isometry and let $A^*$ be the associated induced presentation. Suppose that for each $i$, each component of $A \otimes_X Y_i$ is either a copy of $Y_i$ or is a contractible complex $K$ with $\diam(K) \leq 12 \sys(Y_i)$. Then the natural map $A^* \rightarrow X^*$ has liftable shells.
\end{lem}

The next lemma guarantees convexity. 
\begin{lem}[{\cite[Lem~3.74]{WiseBook}}]\label{lem:convexity of induced}
    Let $X^*$ be a $C'(\frac{1}{14})$ cubical presentation. Suppose $E^* \to X^*$ has no missing shells.
    Let $2\alpha+\beta\leq \frac12$ where $\alpha,\beta>0$. Suppose $|P|_{Y_i} <\alpha \sys(Y_i)$ whenever P is a cone-piece of a translate of $\widetilde Y_j$ in a translate of $\widetilde Y_i$ (so one is not contained in the other).
    Suppose that for any path $P$ in the intersection of translates $\widetilde Y_i,\widetilde E$ in $\widetilde X$, either $|P|_{Y_i} <\beta\sys(Y_i)$ or $Y_i  \subset E$.
    Then $\widetilde E^* \rightarrow \widetilde X^*$ embeds as a convex subcomplex.
\end{lem}

%%%%%%%%%%%%%%%%%%%%%%%%%%%%%%%%%%%%%%%%%%%%%%%
\subsection{Cubulated quotients}
%%%%%%%%%%%%%%%%%%%%%%%%%%%%%%%%%%%%%%%%%%%%%%%
Let $X$ be a nonpositively curved cube complex that is a non-product and where $\pi_1X$ acts without a fixed point at infinity on $\widetilde X$. In Section~\ref{sec: proof without cocompactness} we proved that for every $\alpha>0$ there exists a local isometry $Y\hookrightarrow X$ of a superconvex pseudograph with $\pi_1 Y=  F_2$ such that $\langle X\mid Y\rangle$ is a cubical $C'(\alpha)$ small-cancellation presentation.  We now prove the following, which implies Theorem~\ref{thm:main}. 

\begin{theorem}\label{thm:main with all properties}
        Let $X$ be a nonpositively curved cube complex that admits a local isometry of a rank~$2$ superconvex pseudograph. %Assume that $\pi_1X$ acts without a fixed point at infinity on $\widetilde X$. 
        For every $\alpha\leq \frac 1 {16}$ there is a superconvex rank~$2$ pseudograph $Y\to X$ such that $X^* = \langle X\mid Y\rangle$ is $C'(\alpha)$.

        The complex $Y$ can be chosen so that the following properties hold:
    \begin{enumerate}

        \item\label{conclusions:survival}
        If $S$ is a finite set of nontrivial elements in  $\pi_1X$, then the image $\bar s \in \pi_1X^*$ is nontrivial for each $s \in S$.

        \item\label{conclusions:non-elementary} $\pi_1X^*$ is not virtually cyclic.

        \item\label{conclusions:B6} $X^*=\langle X\mid Y\rangle$ is  $B(6)$. Thus, $\pi_1 X^*$ acts on the CAT(0) cube complex $\mathcal{C}$ dual to the associated wallspace  on $\widetilde X^*$.

        \item\label{conclusions:k-wall convex} $X^*=\langle X\mid Y\rangle$ satisfies $11$-Wall Convexity.

        \item\label{conclusions:free} $\pi_1 X^*$ acts freely  on $\mathcal{C}$.

        \item\label{conclusions:non-product} $\pi_1X^*\backslash\mathcal C$ is a non-product.
        
        \item\label{conclusions:cocompactness preserved} If $X$ is compact then $\pi_1 X^*$ acts cocompactly  on $\mathcal{C}$.      
        
    \end{enumerate}
\end{theorem}

Below we prove Theorem \ref{thm:main with all properties}.\eqref{conclusions:survival}-\eqref{conclusions:k-wall convex} and \eqref{conclusions:non-product}. 
Theorem \ref{thm:main with all properties}.\eqref{conclusions:free} follows from Theorem~\ref{thm:freeness}, and Theorem \ref{thm:main with all properties}.\eqref{conclusions:cocompactness preserved} follows from Theorem~\ref{thm:cocompactness}.

\begin{proof}[Proof of Theorem \ref{thm:main with all properties}.\eqref{conclusions:survival}]

Let $S=\{g_1,\ldots, g_m\}$ be nontrivial elements.
Let $\tilde x\in \widetilde X$ be a basepoint.
For each $k$, let $J_k\rightarrow \widetilde X$ be the convex hull of the lift $[\tilde x,g_k\tilde x]$. Since $S$ is finite,  there is $n \in \naturals$ with $S \subset B_n(\tilde x)$ where $B_n(\tilde x)$ is the radius $n$ ball at $\tilde x$. Choosing $Y$ so that $\sys(Y)> n$ guarantees that if  $X^*=\langle X\mid Y\rangle $, then 
each lift $J_k\rightarrow \widetilde X^*$ is embedded. 
Hence $g_k$ is  nontrivial  in $\pi_1X^*$.
\end{proof}

\begin{proof}[Proof of Theorem \ref{thm:main with all properties}.\eqref{conclusions:non-elementary}]
This holds by the following proposition.
\end{proof}

\begin{prop}\label{prop: Y2 in the quotient by Y1}
    Let $Y\to X$ be a superconvex rank~$2$ pseudograph. Then for any $\alpha\leq \frac{1}{14}$ there exist $Y_1, Y_2\to Y$ such that 
    \begin{itemize}
        \item the presentation $X^* = \langle X\mid Y_1\rangle$ is $C'(\alpha)$,
        \item $Y_2\to X^*$ is $\pi_1$-injective, and
        \item $\widetilde Y_2\to \widetilde X^*$ is an embedding onto a convex subcomplex.
    \end{itemize}
\end{prop}
\begin{proof}
    By Proposition~\ref{prop: no conj} we can assume that the non-diagonal components of $Y\otimes_X Y$ are contractible.
    Let $Y_1, Y_2 \to Y\to X$ be the two superconvex pseudographs of rank~$2$ as in Proposition~\ref{cor:2 pseudographs}, i.e.\ $\langle X\mid Y_1, Y_2\rangle$ satisfies $C'(\alpha)$. The fiber-product $Y_1\otimes_X Y_2$ consists of contractible components. Let $C$ be the maximal diameter of a connected components of $Y_1\otimes_X Y_2$. By possibly replacing $Y_1$ with a further $Y_1'\rightarrow Y_1\rightarrow X$ we can assume $\sys(Y_1)> 12 C$. The presentation $X^* = \langle X \mid Y_1\rangle$ still satisfies $C'(\alpha)$. The induced map of cubical presentations $Y_2 = Y_2^*\to X^*$ has liftable shells by Lemma~\ref{lem:liftable shells condition} (hence has no missing shells). By Theorem~\ref{thm:no_missing_shells_injective} the map $Y_2 \to X^*$ is $\pi_1$-injective and $\widetilde Y_2\to \widetilde X^*$ is an embedding.
     Lemma~\ref{lem:convexity of induced} with $\beta = \alpha$ implies $\widetilde Y_2\to \widetilde X^*$ is an embedding of a convex subcomplex.
\end{proof}

\begin{proof}[Proof of Theorem \ref{thm:main with all properties}.\eqref{conclusions:B6}]
We verify the  $B(6)$ condition for the wallspace structure described in Construction~\ref{const: walspace rank 2}.

\
 \textbf{Wallspace structure:} By construction, each wall is a union of disjoint hyperplanes that separate $Y$, and all hyperplanes  are 2-sided because they are contractible. Each wall is embedded because $\diam(N(U))< \frac{1}{2} \sys(Y)$ for each hyperplane $U$ in $Y$.  Assuming moreover that $\diam(N(U))< \frac{1}{4}sys(Y)$, it also follows that  the disjoint union of the hyperplane carriers corresponding to hyperplanes in the same wall embeds in $Y$.
 
 \
\textbf{Hyperplane convexity:} 
Let $P \rightarrow Y$ be a path that starts on a vertex $p$ and ends on a vertex $q$, so that both $p$ and $q$ lie
on 1-cells dual to a hyperplane $U$ of $Y$, and let $\tau \rightarrow N(U)$ be a path starting on $q$ and ending on $p$. Since $\langle X \mid Y\rangle$ satisfies the $C'(\frac{1}{16})$ condition, the concatenation $P\tau$ is either nullhomotopic or $P$ is the concatenation of at least $17$ pieces. In particular, if $P$ is the concatenation of at most 7 pieces, then $P$ is path homotopic
in $Y$ to a path $P\rightarrow N(u) \rightarrow Y$.

\    
\textbf{Wall convexity:} 
Suppose that $\diam(N(U))\leq \frac{1}{33}\sys(Y)$ and $\diam(K)\leq \frac{1}{33}\sys(Y)$ (which can be ensured by Remark~\ref{rem:diam(K)}). Let $P$ be a path in $Y$ that starts and ends with 1-cells dual to the same wall $U$ of $Y$. 
If $P$ is the concatenation of at most $7$ pieces, then $|P|\leq 7\cdot\frac{1}{16}\sys(Y)$.

If $P$ intersects more than one hyperplane in $U$. First suppose that $P\to Z_i$ where $Z_i$ is as in Construction~\ref{const: walspace rank 2}.
Then $|P|\geq \frac{1}{2}\sys(Y)- \frac{2}{33}\sys(Y)$ since the first and last edges of $P$ are dual to hyperplanes that are also dual to antipodal edges of $\sigma_i$ where $\sigma_i$ is a closed path in of $Z_i$ realizing its systole (as in Construction~\ref{const: walspace rank 2}).
Now suppose that the first edge of $P$ belongs to $Z_1$ but not $Z_2$, and the last edge of $P$ belongs to $Z_2$ but not $Z_1$. Then $P$ has a an initial subpath $P'$ that ends in $K$. Then, similarly as above $|P|\geq |P'|\geq \frac{1}{2}\sys(Y) - \frac{2}{33}\sys(Y)$.

Thus $\frac{1}{2}- \frac{2}{33} < 7\cdot\frac{1}{16}$, which is a contradiction. Thus $P$ starts and ends on the same hyperplane $u$ in $U$, and by hyperplane convexity (see above), $P$ is path-homotopic into $N(u)$.

\textbf{Equivariance:} This holds since $\Aut_X(Y)$ is trivial by the choice made in Construction~\ref{const: walspace rank 2}. See Remark~\ref{rem: trivial aut}. 
\end{proof}

\begin{proof}[Proof of Theorem~\ref{thm:main with all properties}.\eqref{conclusions:k-wall convex}]

The proof of wall convexity in \ref{thm:main with all properties}.\eqref{conclusions:B6} generalizes to an arbitrary $k$ assuming that $\alpha<\frac{1}{2k}$ and $\sys(Y)$ is sufficiently large compared to $\diam(N(U))$ and $\diam(K)$.
\end{proof}

\begin{proof}[Proof of Theorem \ref{thm:main with all properties}.\eqref{conclusions:non-product}]
By Proposition~\ref{prop: Y2 in the quotient by Y1} there exist superconvex rank $2$ pseudographs $Y_1, Y_2\to X$ such that $Y_2\to X^*$ is $\pi_1$-injective where $X^* = \langle X \mid Y_1\rangle$, and $\widetilde Y_2 \to \widetilde X^*$ is an embedding as a convex subcomplex.  By Part~\eqref{conclusions:B6} we can assume $X^*$ is $B(6)$.
We will show the dual $\mathcal C$ of $X^*$ is not a product. By  Proposition~\ref{prop: non-product characterization}, it suffices to show $\mathcal C$ has a pair of strongly separated hyperplanes.

Recall that by Lemma~\ref{lem:bounded wall pieces}, the diameter of wall-pieces of $Y$ is bounded by some constant $M'$. Since $Y_1\otimes_X Y_2$ has contractible components, there exists a constant $M''$ bounding the diameter of the cone pieces between $Y_1$ and $Y_2$. Let $M=\max\{M',M''\}$. Note that for any $\Sigma\to Y_2$ the cone-pieces between $\Sigma$ and $Y_1$ are bounded by $M$.
Since $Y$ us a compact pseudograph, there is also a bound $B$ on the diameter of hyperplanes of $Y$. 

Let $U_1, U_2$ be strongly separated hyperplanes in $\widetilde Y_2$  at distance greater than the maximum of $7\max\{\frac{sys(Y_1)}{\alpha}, M\}$ and $2M+B$.
Consider the hyperplanes in  $\widetilde X^*$ extending $U_1, U_2$, and continue to denote them by $U_1, U_2$.
We will prove that $U_1,U_2$ lie in distinct walls $W_1, W_2$ are strongly separated (meaning $W_1, W_2$ do not cross and no other wall crosses both of them). 

Let $\sigma$ be a geodesic path connecting $U_1$ to $U_2$ in $\widetilde X^*$. 
As $\widetilde Y_2$ is convex,  $\sigma\rightarrow \widetilde X^*$ lies in $\widetilde Y_2$. 
Let $\Sigma$ be the cubical convex hull of $\sigma$ in $\widetilde X$. Note that $\Sigma$ is contractible. 
By possibly replacing $Y_1$ by another superconvex rank~$2$ pseudograph mapping into $Y_1$, we can assume $\sys(Y_1)> \alpha \diam(\Sigma)$.

Consider the following wallspace on $\Sigma$: there is a noteworthy wall $U$ consisting of the two hyperplanes $U_1\cap \Sigma$ and $U_2\cap \Sigma$, but every other wall consists of a single hyperplane of $\Sigma$.
The cubical presentation
$X^*_{\Sigma}=\langle X \mid Y_1, \Sigma\rangle$ has the same $\pi_1$ as $X^*$, but has a coarser wallspace structure. We use  $X^*_{\Sigma}$ to facilitate  the proof of the strong separation of $W_1,W_2$. We will denote the wall of $X^*_{\Sigma}$ containing $U$ by $W$.

Since $X^*$  satisfies $C'(\alpha)$, so does $X^*_{\Sigma}$. Indeed, since $\sys(Y_1)> \alpha \diam(\Sigma)$, then $\alpha \diam(Y_1\cap \Sigma)< \sys(Y_1)$, and the pieces coming from $X^*$ still satisfy this condition in $X^*_{\Sigma}$. 
The $B(6)$ condition (using the  wallspace on $\Sigma$ described above, and the standard structure on $Y_1$) is also satisfied by $X^*_{\Sigma}$. Indeed, $Y_1$ already satisfies all the hypothesis of  $B(6)$; using the contractibility of $\Sigma$,  it is straightforward that the choice of walls on $\Sigma$ satisfies Hypothesis~\eqref{item:b6.2} of the $B(6)$ condition, and, again since $\Sigma$ is contractible, hyperplane convexity and Hypothesis~\eqref{item:b6.4} of the $B(6)$ condition follow vacuously. To finish verifying the $B(6)$ condition, we need to check the Wall Convexity condition only for the walls in $\Sigma$, since this condition was verified for the walls in $Y_1$ in the proof of 
Part~\eqref{conclusions:B6} of this theorem.
Let $S$ be a path in $\Sigma$ that starts and ends with 1-cells
dual to the same wall of $\Sigma$. Suppose $S$ is the concatenation of at most 7 pieces. The diameter of a piece in $\Sigma$ is at most $\max\{\frac{sys(Y_1)}{\alpha},M\}$ where $M'$ is the bound on wall-pieces from Lemma~\ref{lem:bounded wall pieces}. Since $U_1, U_2$ were chosen at distance $>7\max\{\frac{sys(Y_1)}{\alpha},M\}$,  it follows that the path $S$ starts and ends on the same hyperplane of $\Sigma$, and is thus nullhomotopic.
   
Let $\mathcal{C}_\Sigma$ be the cube complex dual to the wallspace on $\widetilde X^*_{\Sigma}$ described above and note that there is a natural projection $\mathcal{C} \to\mathcal{C}_\Sigma$.
By construction of $\mathcal{C}_\Sigma$ the hyperplanes of $\mathcal{C}$ dual to $W_1, W_2$ project to a single hyperplane  of $\mathcal{C}_\Sigma$. Since a hyperplane in a CAT(0) cube complex cannot self-cross, we deduce that $W_1,W_2$ cannot cross in $\widetilde X^*$.

We now argue that $W_1$ and $W_2$ are strongly separated in $\widetilde X^*$.
Suppose there exist a wall $W'$ in $\widetilde X^*$ that crosses both walls $W_1$ and $W_2$. We claim that $W'$ crosses $\Sigma$ in a hyperplane that we denote by $U'$. Let $V_1, V_2$ be hyperplanes of cone-cells of the wall $W$ (which is the wall of $\widetilde X^*_{\Sigma}$ ``containing'' $W_1, W_2$) crossed by $W'$. 
By Proposition~\ref{prop:every_cone_between}, every cone of $W$ between $V_1$ and $V_2$ is crossed by $W'$. In particular $W'$ crosses $\Sigma$ which separate $W_1,W_2$ in $W$.
Moreover, considering the $\Sigma$ cone-cell in the ladder, we see that $U'\cap \Sigma$ and $U_i\cap \Sigma$ either cross each other (which  happens if 
the first or last cone-cell maps to  $\Sigma$), or are dual to edges in a single piece of $\langle X\mid Y_1, \Sigma\rangle$ which is either a  wall-piece of $\Sigma$, or  a  cone-piece of $\Sigma$ with $Y_1$.
Thus, $\dist(U_1, U_2)\leq \dist(U_1, U')+\dist(U_2, U')+\diam(U')\leq 2M+B$. But by our choice of $U_1, U_2$ at distance $\dist(U_1, U_2)> 2M+B$, which yields a contradiction.

Finally,  $\mathcal{C}$ is not a quasiline  by Part~\eqref{conclusions:non-elementary}.
\end{proof}

Proposition~\ref{prop:every_cone_between}  is the remaining ingredient needed in the proof of Theorem \ref{thm:main with all properties}.\eqref{conclusions:non-product}. 
It engages with the following notions:

A \emph{$W$-ladder} $L$ is a ladder mapping to $\widetilde X^*$
such that 
\begin{itemize}
\item $L$ is the union of $2$-cells $C_1,C_2,\ldots, C_m$,
\item each $C_i$ is  a square or a cone-cell,
\item $C_i\cap C_{i+1}$ is a 1-cube $e_i$, and
\item each $e_i$ is dual to a hyperplane of $W$.
\end{itemize}

\begin{rem}[Squares as cones]\label{rem:add squares}
Let $X^*=\langle X \mid Y_1, Y_2, \ldots, \rangle$ be a cubical presentation that satisfies $C'(\alpha)$.
Let $s$ be a 2-cube of $X$. If $s$ does not lie in any $Y_i$, then we can add $s$ to the cubical presentation
to obtain
$\langle X \mid s, Y_1,Y_2,\ldots, \rangle$
which satisfies the same $C'(\alpha)$ condition.
Indeed, $\sys(s)=\infty$, and if  $s$ is not  in any $Y_i$,
 any piece between $s$ and any $Y_j$ is already a wall piece. It follows that we can add as many squares as needed to $X^*$ to form a $C'(\alpha)$ cubical presentation $X^*_\square$ where every square lies in a cone. This will facilitate the use of the Ladder Theorem in Proposition~\ref{prop:every_cone_between}. 
\end{rem}

\begin{prop}[Crossing Intermediate Cones]\label{prop:every_cone_between}
Let $X^*$ be $B(6)$.
Let $W,W'$ be walls in $\widetilde X^*$.
Assume there are distinct hyperplanes or cones  $V_1,V_2$ of $W$ that are crossed by $W'$. Then every cone $Y$ of $W$ between $V_1$ and $V_2$ is 
crossed by $W'$.

Moreover, there is a
ladder $L$ in $X^*_\square$,
containing a $W$-ladder $K$,
and a $W'$-ladder $K'$.
And $K,K'$ start and end on the first and last cells of $L$.
And $L$ is \emph{quasi-2-collared} in the sense that every external cell of $L$ is a cell of  $K$ or $K'$ or both. See Figure~\ref{fig:2collaredladder}.

\begin{figure}
    \centering
\includegraphics[width=0.8\linewidth]{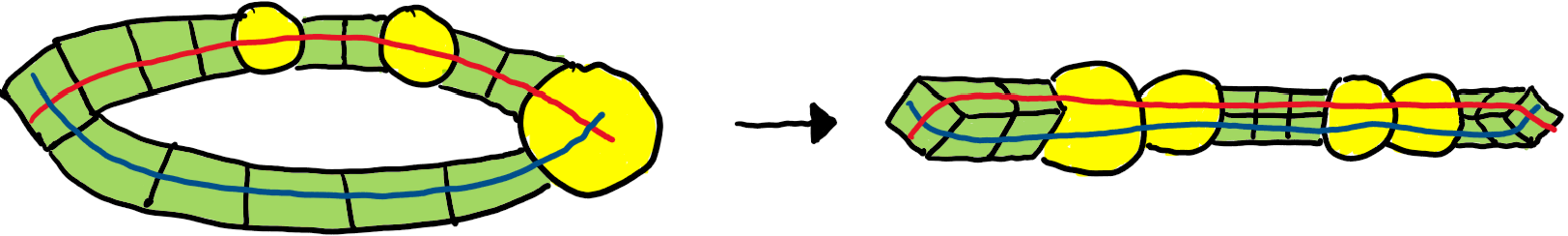}
    \caption{A $2$-collared diagram exhibiting two crossing between walls $W$ and $W'$ on the left, and the corresponding quasi-$2$-collared  ladder on the right.}
\label{fig:2collaredladder}
\end{figure}
\end{prop}

\begin{remark}\label{rmk: reduced moves}
    In the proof of Proposition~\ref{prop:every_cone_between} we use the notion of a reduced diagram. A diagram is \emph{reduced} if it cannot be simplified by performing any of the following moves: removing square bigons, combining cone-cells, absorbing a square into a cone-cell, absorbing a cornsquare into a cone-cell, and replacing internal non-essential cone-cells with square diagrams. See the discussion following~\cite[Def~3.11]{WiseBook}.
\end{remark}

\begin{proof}
Let $K$ be a reduced $W$-ladder in $\widetilde X^*$ that starts and ends with a square or cone-cell
containing the intersection at $V_1$ and $V_2$.

Viewing $K$ as a $W$-ladder in $\widetilde X^*_\square$,  we break $K$ into a sequence of sub-$W$-ladders
$K_1,K_2,\ldots, K_m$
where for each $i$ the first square/cone-cell of $K_{i+1}$ is the last square/cone-cell of $K_{i}$, 
and where the squares/cone-cells of $K_i$ crossed by $W'$ are precisely the first and last ones. Since we are working in $\widetilde X^*_\square$, we treat the first and last cells of each $K_i$ as cone-cells.

We claim that for each $i$ there is a diagram $L_i$ that is $2$-collared by $K_i$ and a $W'$-ladder $K_i'$.
Let $K_i'$ be any $W'$-ladder starting and ending at the first and last cone-cells of $K_i$.
Let $J_i$ be the union of $K_i$ and $K_i'$ along their first and last cone-cells. We can assume without loss of generality that the combined cone-cells of $K_i, K_i'$ are the same, since any two paths in a cone are subpaths of a closed path in the cone.  
Thus $J_i$ is an annulus or Möbius strip.

Let $D_i$ be a diagram whose boundary path $P_i$ is an immersed path in $J_i$  generating $\pi_1 J_i$ and where $P_i$ does not traverse any edge of $K_i$ dual to $W_i$. 
Moreover choose $K_i'$ and $D_i$  so that $D_i$ has minimal complexity among all such choices.
If $P_i$ traversed an edge of $K_i'$ dual to $W_i'$, then following this dual curve within $D_i$, we see that it cannot exit on the $K_i$ side by our non-intersection assumption on $K_i$, and hence it exits on the $K_i'$ side. This contradicts the minimality of $D_i$.
Thus $P_i$ is a boundary cycle of $J_i$, and $J_i$ is not a Möbius strip.
The union $L_i = J_i\cup_{P_i} D_i$ is a disc diagram, proving the claim.

We now amalgamate $L_1,L_2,\ldots, L_m$ by combining cone-cells mapping to the same cone, using that any two paths in a cone are subpaths of a closed path. 
See Figure~\ref{fig:combined}.

\begin{figure}
    \centering    \includegraphics[width=1\linewidth]{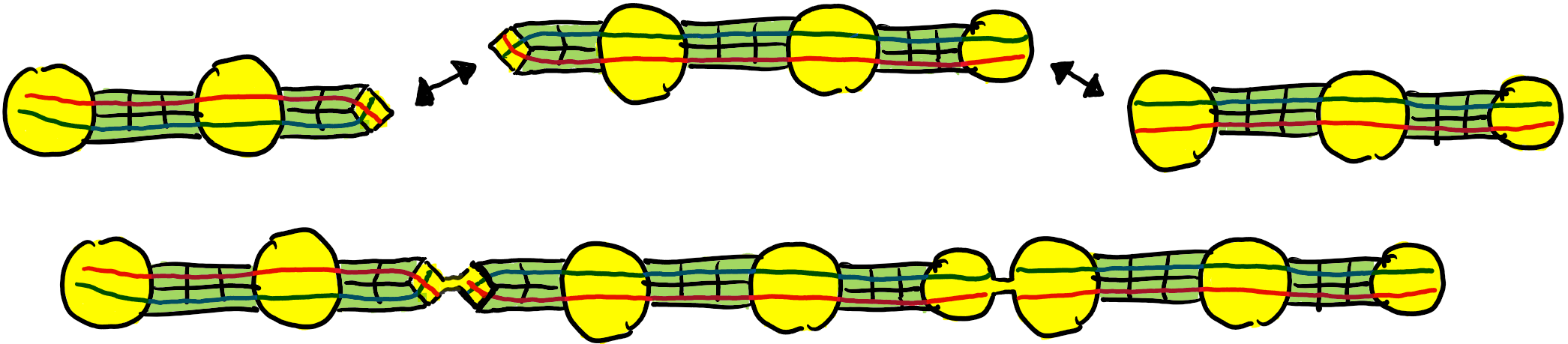}
    \caption{Combining  ladders.}
    \label{fig:combined}
\end{figure}

Let $\bar L$ be the combination of $L_1,L_2,\ldots, L_m$.
We obtain $L$ from $\bar L$ by reducing to obtain a  disc diagram using the reduction moves in Remark~\ref{rmk: reduced moves}.
These moves preserve the quasi-2-collaring by $W$ and $W'$.
%See the discussion following~\cite[def~3.11]{WiseBook} for the description of reduction moves.
Moreover, our hypothesis that $W'$ only crosses at the first and last cone-cells of each $K_i$ ensures that $K$ is transformed to a reduced ladder in $X^*_\square$ visiting the same sequence of cones as the original $W$-ladder.
\end{proof}

%%%%%%%%%%%%%%%%%%%%%%%%%%%%%%%%%%%%%%%%%%%%%%%%%%%%%%%%%%%%%%%%%%%%%
%%%%%%%%%%%%%%%%%%%%%%%%%%%%%%%%%%%%%%%%%%%%%%%%%%%%%%%%%%%%%%%%%%%%%
\section{Freeness}\label{sec:freeness}
%%%%%%%%%%%%%%%%%%%%%%%%%%%%%%%%%%%%%%%%%%%%%%%%%%%%%%%%%%%%%%%%%%%%%
%%%%%%%%%%%%%%%%%%%%%%%%%%%%%%%%%%%%%%%%%%%%%%%%%%%%%%%%%%%%%%%%%%%%%
In this section, we prove Part \eqref{conclusions:free} of Theorem~\ref{thm:main with all properties}, i.e.\ that the pseudograph $Y\to X$ can be chosen so that the action of $\pi_1 X^*$ on $\mathcal C$ is free. 
We start by recalling a criterion that is often useful for proving properness  in the setting of cubical small-cancellation.

A hyperplane $U$ is \emph{$m$-proximate} to a $0$-cube $v$ if there is a path $P =P_1, \ldots, P_m$
such that each $P_i$ is either a single edge or a piece, $v$ is the initial
vertex of $P_1$ and $H$ is dual to an edge in $P_m$. A wall is \emph{$m$-proximate} to $v$ if it
has a hyperplane that is $m$-proximate to $v$. 

 A hyperplane $U$ of a cone over $Y$ is \emph{piecefully convex} if the following holds: For any path $\tau\rho \rightarrow Y$ with endpoints on $N(U)$, if $\tau$ is a geodesic and $\rho$ is trivial or lies in a piece of $Y$ containing an edge
dual to $U$, then $\tau\rho$ is path-homotopic in $Y$ to a path $\mu \rightarrow N(U)$.

\begin{theorem}[{\cite[Thm~5.44 and Cor~5.45]{WiseBook}}]\label{thm:proper}
 Let $X^*= \langle X \mid \{Y_i\} \rangle $ be a  $B(6)$ cubical presentation that satisfies the following conditions:
 \begin{enumerate}
\item \label{item:prop.1} Each hyperplane $H$ of each cone over $Y_i$ is piecefully convex.
\item \label{item:prop.2} Let $\kappa \rightarrow Y \in \{Y_i\}$ be a geodesic with endpoints $p, q$. Let $U_1$ and $U'_1$ be distinct hyperplanes in the same wall $W_1$ of $Y$. Suppose $\kappa$ traverses a 1-cell dual to $U_1$, and either $U'_1$ is 1-proximate to $q$ or $\kappa$ traverses a 1-cell dual to $U'_1$. Then there is a wall $W_2$ in $Y$ that separates $p, q$ but is not
2-proximate to $p$ or $q$.
\item \label{item:prop.3} Each infinite order element of $\Aut_X(Y_i)$ is cut by a wall.
 \end{enumerate}

Then $\pi_1X^*$ acts on the dual CAT(0) cube complex  $\mathcal C$ with torsion stabilizers.
\end{theorem}

The following result shows that if    $Z\rightarrow X$ is homotopy equivalent to a circle, then $Z$  can be chosen so that the cubical presentation $X^*=\langle X \mid Z \rangle$  satisfies the conditions of Theorem~\ref{thm:proper}. We state it in a way that will simplify the proof of Lemma~\ref{lem: wallspace structure freeness} (stated below).

\begin{theorem}[{\cite[Thm~3.5]{FuterWise}}]\label{thm:FuterWise}
    Let $X^*= \langle X \mid Z \rangle $ be a cubical presentation with $Z$ homotopy equivalent to a circle, and satisfying the $C'(\alpha)$ condition for $\alpha\leq \frac{1}{20}$ and where $\diam (N(U)) < \alpha \sys(Z)$ for every hyperplane of  $Z$.  Then  $X^*= \langle X \mid Z \rangle $ satisfies the conditions of Theorem~\ref{thm:proper}. Moreover, for every geodesic $\kappa \rightarrow Z$:
    $$
   |\kappa| < \frac{1}{2}\sys(Z) + 2\alpha \sys(Z).$$
\end{theorem}

\begin{remark}
   Theorem~\ref{thm:FuterWise} as stated in~\cite{FuterWise} assumes moreover that $Y - U$ is contractible for each hyperplane $U$ of $Y$. The reason for this assumption is that then the wallspace structure in Construction~\ref{const: rank 1 wall} is slightly simplified, since all hyperplanes in $Y$ must then cross the geodesic $\sigma$ in the construction. However, this hypothesis is unnecessary for the calculations therein, so we omit it.
\end{remark}

\begin{lemma}\label{lem: wallspace structure freeness}
   Let $X^*=\langle X \mid Y \rangle $ be a cubical presentation that satisfies the $C'(\alpha)$ condition for $\alpha\leq\frac{1}{16}$, and where $Y$ is a superconvex rank~$2$ pseudograph that is the union of two rank~$1$ pseudographs  $Z_1$ and $Z_2$, which are locally convex in $Y$. 
   %Let $M=\min_i\{sys(Z_i)\}$ and  $\diam (Z_1 \cap Z_2)\leq \beta M$ where $\alpha + \beta <\frac{1}{2}$.
   Let $M=\min_i\{sys(Z_i)\}$ and  $\diam (Z_1 \cap Z_2)\leq \beta M$ where $\alpha + \beta <\frac{1}{8}$. 
Suppose $\diam (N(U))\leq \alpha M$ for each hyperplane $U$ of $Y$.

   Consider the wallspace structure on $Y$ obtained by combining the antipodal wall structure on $Z_1$ and $Z_2$ as in Construction~\ref{const: walspace rank 2}. Then this wallspace structure satisfies the hypotheses of Theorem~\ref{thm:proper}. Therefore, the action of $\pi_1 X^*$ on the dual $\mathcal{C}$ to this wallspace structure has torsion stabilizers.
\end{lemma}

\begin{proof}
 The  $B(6)$ condition is verified in Part~\ref{conclusions:B6} of  Theorem~\ref{thm:main with all properties}. We now verify the rest of the hypotheses in Theorem~\ref{thm:proper}.
 
\
\textbf{Condition~\ref{item:prop.1} (Pieceful-convexity):} 
Let $\xi\rho \rightarrow Y$ be a path with endpoints on $N(U)$, and where $\xi$ is a geodesic and $\rho$ is either trivial or lies in a piece of
$Y$ containing an edge dual to $U$. The $C'(\alpha) $ condition implies that the diameter of pieces appearing in essential paths in $Y$ is  $< \alpha M$. Let $\tau \rightarrow N(U)$  be a geodesic path joining the endpoints of $\xi\rho$ in $N(U)$. Since $\diam (N(U))\leq \alpha M$, and $|\xi|< \frac{1}{2}M + 2\alpha M$ by Theorem~\ref{thm:FuterWise}, then $|\xi\rho\tau|< \frac 1 2 M+4\alpha M < M$ provided that we choose $\alpha$ so that $4\alpha <\frac 1 2$. Since the concatenation $\xi\rho\tau$ is shorter than the systole of $Y$, then $\xi\rho\tau$ must be nullhomotopic, so $\xi\rho$ is homotopic into the carrier of $U$, as claimed. 

\
\textbf{Condition~\ref{item:prop.2} (Existence of non-proximate separating wall):}
 Let $\kappa$ be a geodesic path from $p$ to $q$ as in  Theorem~\ref{thm:proper}.\eqref{item:prop.2} and  consider distinct hyperplanes $U_1, U_1'$ of $W_1$ where $\kappa$ traverses an edge $e$ dual to $U_1$ and either $\kappa$ traverses an edge dual to $U_1'$, or $q$ is 1-proximate to an edge dual to $U_1'$. If $U_1$ and $U_1'$ are hyperplanes of the same $Z_i$, then the result follows from Theorem~\ref{thm:FuterWise}.

 Otherwise, the wall $W_1$ consists of  hyperplanes $U_1$, $U_1'$ and $U_1''$ where $U_1'' \subset Z_1\cap Z_2$. Without loss of generality assume that $U_1\subseteq Z_1$ and $U_1'\subseteq Z_2$. We claim that  $\kappa \cap Z_1\cap Z_2 \neq \emptyset$.
 This is clear if $\kappa$ traverses an edge dual to $U_1'$. If $q$ is 1-proximate to an edge dual to $U_1'$, then $q$ is at distance at least $\frac{1}{2}M -\alpha M$ from $U_1''$. Thus $q$ is at distance at least $\frac{1}{2}M - \alpha M - \beta M$ from $Z_1\cap Z_2$, which is greater than $\frac {3}{8} M$, since $\alpha+\beta<\frac{1}{8}$. This proves that $\kappa$ must intersect $Z_1\cap Z_2$. 
 
 Let $o\in \kappa \cap Z_1\cap Z_2$.
 Note that $p$ is at distance at least $\frac 38 M$ from $o$ by a similar argument to the above. Let $e_2$ be an edge on $\kappa$ halfway between $o$ and $p$ and let $U_2$ be the hyperplane dual to $e_2$. Then the distance from $U_2$ to each of $p,q$ is at least $\frac 3{16} M$, in particular $U_2$ is not $2$-proximate to $p$ or $q$. The same holds for the hyperplane $U_2'$ dual to the edge antipodal to  $e_2$ in $Z_1$. Thus the wall $W_2$ containing $U_2$ satisfies the condition.% of Theorem~\ref{thm:proper}. 

\textbf{Condition~\ref{item:prop.3} (Cutting infinite-order automorphisms):} This holds  since  $Y$ is compact so $\Aut_X(Y)$ is finite (and in our case, trivial by Remark~\ref{rem: trivial aut}).
\end{proof}

We can now prove  freeness of the action of $\pi_1X^*$ on $\mathcal{C}$.

\begin{thm}\label{thm:freeness}
Let $X$ be a nonpositively curved cube complex that admits a local isometry $Y\rightarrow X$ of a superconvex rank~$2$ pseudograph, where $X^*=\langle X \mid Y \rangle $ satisfies $C'(\alpha)$ for some $\alpha \leq \frac{1}{16}$ and satisfies  $B(6)$. Then there exist superconvex rank~$2$ pseudographs $Y'\to Y$ such that $\pi_1 X^{**}$ acts freely on the dual $\mathcal{C}$ of the wallspace structure on $\widetilde X^{**}$ where $X^{**} = \langle X \mid Y'\rangle$.
\end{thm}

\begin{proof}
By Construction~\ref{const: walspace rank 2} there exist $Y'\to Y$ which is the union of two rank~$1$ pseudographs $Z_1$ and $Z_2$, and a contractible subcomplex $K\subset Y'$, which all are locally convex in $Y'$, and such that $Z_1\cap Z_2 \subset K$. To simplify the proof, we assume that $Z_1\cap Z_2 = K$ (which can be done by replacing $Z_i$ with $Z_i\cup K$). 

Since $Y$ is compact, there exists a constant $D>0$ such that for each hyperplane $U$ in $Y$, $\diam N(U)<D$. The same follows for $Y'$. 
We note that in  Construction~\ref{const: walspace rank 2} the systoles $\sys(Z_i)$ can be chosen to arbitrarily large, so we can assume that $\sys(Z_i)> 16\max\{D, \diam(Z_1\cap Z_2)\}$. This ensures that the assumptions of Lemma~\ref{lem: wallspace structure freeness} are satisfied with $\alpha = \beta \leq\frac{1}{16}$, so $\pi_1X^{**}$ acts on the dual cube complex $\mathcal{C}$ with torsion stabilizers. 

We now explain that  $\pi_1X^{**}$ is torsion-free and therefore the cell-stabilizers are trivial. This can be deduced in two ways. It follows from~\cite[Thm~4.2 and Rmk~4.3]{WiseBook} that if $X^{**}$ is $C'(\frac{1}{20})$, then every torsion element in $\pi_1X^{**}$ has to be conjugated into $\Aut_X(Y')$, which is trivial by Remark~\ref{rem: trivial aut}, so the cell stabilizers in the dual are trivial. Alternatively,  using the $C(9)$ condition, and   $\Aut_X(Y')=\{1\}$, it follows from the homology formula in~\cite[Cor 1.6]{Arenas24asph} that $\pi_1X^{**}$ is torsion-free.
\end{proof}

%%%%%%%%%%%%%%%%%%%%%%%%%%%%%%%%%%%%%%%%%%%%%%%%%%%%%%%%%%%%%%%%%%%%%

\section{Cocompactness}\label{sec:cocompact}
%%%%%%%%%%%%%%%%%%%%%%%%%%%%%%%%%%%%%%%%%%%%%%%%%%%%%%%%%%%%%%%%%%%%%
%%%%%%%%%%%%%%%%%%%%%%%%%%%%%%%%%%%%%%%%%%%%%%%%%%%%%%%%%%%%%%%%%%%%%

The goal of this section is the following which is proven in Section~\ref{subsec:proof of cocompactness}:

\begin{thm}\label{thm:cocompactness}
Let $X^*=\langle X \mid Y \rangle $ be a compact cubical presentation that satisfies the $C'(\frac{1}{12})$ and 
 $B(6)$ conditions, $11$-wall convexity, and where  $Y$ is a superconvex rank~$2$ pseudograph.
Then  $\pi_1 X^*$ acts 
cocompactly on  the dual $\mathcal{C}$ of the wallspace structure on $\widetilde X^*$.
\end{thm}

The proof of Theorem~\ref{thm:cocompactness} involves thin triangles (Section~\ref{sec: wall traingles}), induced subwallspaces and cubical presentations (Section~\ref{subsec:induced}), hemiwallspaces (Section~\ref{subsec: hemi}),  and relative hyperbolicity (Sections~\ref{sec: _relative hyperbolicty} and \ref{sec: relative hyperbolicty}).

\subsection{Wall-triangles}\label{sec: wall traingles}

This section introduces wall-triangles, which are disc diagrams determined by  triples of pairwise intersecting walls. We prove a technical result, Proposition~\ref{prop: uniformly thin wall triangles}, which is a variant of the Ladder Theorem for  $B(6)$ cubical presentations. 
Corollary~\ref{cor: existence of r}, which follows from Proposition~\ref{prop: uniformly thin wall triangles}, will be used to prove Theorem~\ref{thm:cocompactness}. 

Let $W_1, W_2, W_3$ be walls in $\widetilde X^*$.
A \emph{wall-triangle collared by $W_1, W_2, W_3$} is a 3-collared disc diagram $D$ (see Definition~\ref{defn: collared}) where the collaring ladders $L_1, L_2, L_3$ carry dual curves of $W_1, W_2, W_3$. 
We refer to $L_i$ as the \emph{$W_i$-collar} of $D$.
The ladders $L_1, L_2, L_3$ intersect at \emph{corner-cells} $C_{12}$, $C_{23}$, $C_{31}$, which are cone-cells or squares.
See Figure~\ref{fig:wall-triangles}.
%We refer to $D$ as \emph{collared by} $W_1, W_2, W_3$.
The wall-triangle $D$ is \emph{minimal} if $D$ has minimal complexity among all wall-triangles collared by $W_1, W_2, W_3$.

Let $U_1$ be a hyperplane of $W_1$. Assume $L_1\to N(U_1)$ is a square ladder. We say $D$ is \emph{$U_1$-minimal}, if $D$ has a minimal complexity among all wall-triangles collared by $W_1, W_2, W_3$ where the $W_1$-collar is a square ladder mapping to $N(U_1)$. We do not claim that a $U_1$-minimal wall-triangle is necessarily minimal in the sense above, but the statements that we prove below also apply to $U_1$-minimal wall triangles, and will be used in the proof of Theorem~\ref{thm:cocompactness}.

\begin{figure}
\centering
\includegraphics[height=1in]{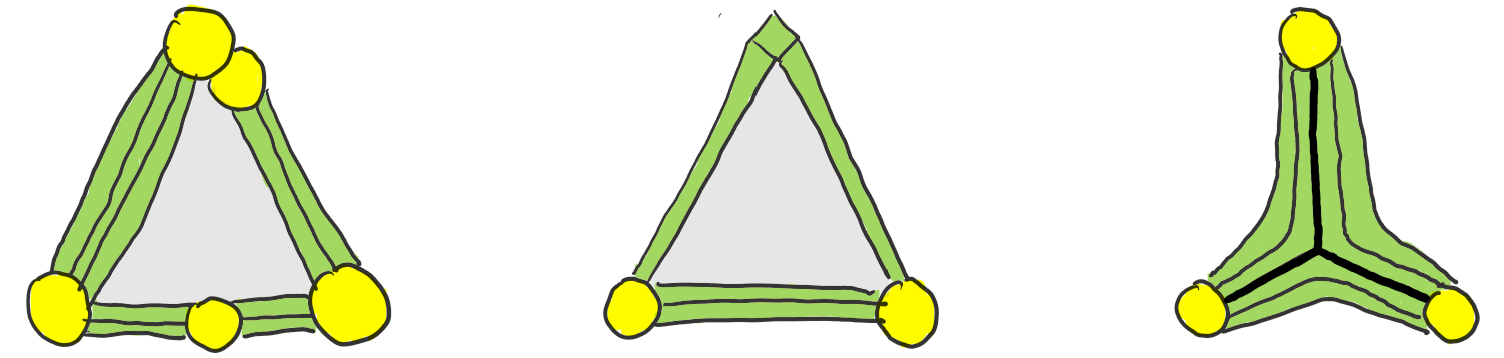}
\caption{Three wall-triangles.
\label{fig:wall-triangles}}
\end{figure}

\begin{lem}\label{lem: collared diagrams for crossing walls}
If three walls pairwise cross in $\widetilde X^*$, then there is a wall-triangle $D \rightarrow \widetilde X^*$ collared by these walls.

%3-collared diagram $D \rightarrow \widetilde X^*$ whose collaring ladders carry dual curves mapping to these walls.
\end{lem}

\begin{proof}
Let $W_1,W_2, W_3$ be pairwise crossing walls. 
Consider associated ladders $L_1, L_2, L_3$, carrying a dual curve in each of $W_1,W_2, W_3$, and with the crossings occurring at the first and last cells of each $L_i$. 
Consider the union $A=L_1\cup L_2\cup L_3$. By construction $A$ is homotopy equivalent to a cycle. 
 
Let $P\to A$ be a combinatorial embedding of a closed cycle $P$, that generates $\pi_1A$.
Consider a disc diagram $D_0$ with $\partial D_0 = P$. If $P$ traverses an edge dual to any of $L_1, L_2, L_3$, then there is a proper subdiagram $D_0'\subseteq D_0$ with $\partial D_0'$ contained in a union of ladders $L_1', L_2', L_3'$ carrying dual curves of $W_1,W_2, W_3$. Note that $P$ must traverse such an edge if $A$ is a Möbius strip.
Thus, if $D_0$ has minimal complexity among all the disc diagrams obtained this way for $W_1, W_2, W_3$, i.e.\ among all choices of $L_1, L_2, L_3$, $P$, and $D$, then $P$ does not traverse any such edge. In particular, $A$ is an annular diagram, and so $D =D_0\cup A$ is a wall-triangle collared by $W_1,W_2, W_3$.   
\end{proof}

\begin{prop}\label{prop: uniformly thin wall triangles}
Let $X^*=\langle X \mid \{Y_i\} \rangle $ be a 
 cubical presentation satisfying the $C'(\frac{1}{12})$ and
 $B(6)$ conditions, and $11$-wall convexity.
 \begin{enumerate}
     \item\label{part: wall-triangles} Each cone-cell in a minimal (or $U_1$-minimal) wall-triangle is a corner-cell.
     \item\label{part: bounded} 
     If $X^*$ is compact, then there exists $R>0$ such that for any walls $W_1, W_2, W_3$ in $\widetilde X^*$, the minimal (or $U_1$-minimal) wall-triangle collared by $W_1, W_2, W_3$ has at most $R$ hyperplanes.

 \end{enumerate}
\end{prop}

\begin{proof}[Proof of (\ref{part: wall-triangles})]
Let $D$ be a minimal wall-triangle collared by $W_1,W_2, W_3$. Let $L_1, L_2, L_3$ be its collaring ladders, and let $C_{12}, C_{23}, C_{31}$ be its corner-cells.  

In the  degenerate cases where $ D$ is a cone-cell or a ladder, the statement is clear.  
We focus on the non-degenerate case where the corner-cells are distinct.

Suppose that $D$ contains an interior cone-cell $C$, i.e.\ $C$ is disjoint from $\partial D$. Consider the cyclic order on the dual curves starting at $\partial C$ and travelling away from $C$. By $C'(\frac{1}{12})$ and Lemma~\ref{lem:Greendlinger} no dual curve returns to $\partial C$, and therefore must end on $\partial D$. Any two dual curves that intersect in $D$ must start in a single piece in $\partial C$. Thus by $C'(\frac{1}{12})$ and the pigeon-hole principle, there exists a pair of dual curves which end at the same side ladder $L_i$ such that any subpath of $\partial C$ containing both of them has at least $5$ pieces. Thus the subdiagram $D'$ bounded by a subladder of $L_i$, the carriers of these two dual curves, and $C$, is a disc diagram with only two exposed cells, and therefore by Lemma~\ref{lem:Greendlinger}, $D'$ must be a ladder. Hence $C$ lies in $L_i$, which contradicts the assumption that $C$ is an interior cone-cell. Thus there are no interior cone-cells in $D$.

Now suppose that $D$ contains a non-interior cone-cell $C$, which is different from $C_{12}, C_{23}, C_{31}$. Without loss of generality, assume that $C$ is contained in $L_2$. By $11$-wall convexity the innerpath $S$ of $\partial C$ consists of at least $11$ pieces. Again we consider all the dual curves that start at $S$ and again none of them return to $S$. Any dual curve that ends in $L_2$ must belong to the initial or terminal piece of $S$. By the pigeon-hole principle (applied to the second, sixth, and tenth pieces of $S$) there exists a pair of dual curves that both end on $L_i$  for $i=1$ or $3$, such that the subpath of $S$ containing both of them has at least $5$ pieces. Thus the subdiagram $D'$ bounded by a subladder of $L_i$, the carriers of these two dual curves, and $C$, is a disc diagram with only two exposed cells. By Lemma~\ref{lem:Greendlinger}, $D'$ is a ladder. Hence $C$ belongs to both $L_2$ and $L_i$. This contradicts the minimality of $D$, as the subdiagram of $D$ with cone-cells $C$ replacing $C_{12}$ or $C_{31}$ depending on whether $i=1$ or $3$ is a smaller diagram collared by $W_1, W_2, W_3$. 

For the $U_1$-minimal case, we must have $i=3$ in the above argument, and we get a contradiction with the $U_1$-minimality.
\end{proof}

\begin{proof}[Proof of (\ref{part: bounded})]
 Let $D$ be a minimal (or $U_1$-minimal) wall-triangle collared by $W_1, W_2, W_3$. 
 By Part~\eqref{part: wall-triangles} $D$ does not have any cone-cells that are not corner-cells. Consequently, the collaring ladders $L_i$ decompose as the union $C_{(i-1)i}\cup G_i\cup C_{i(i+1)}$ where $G_i$ is a square ladder. The hyperplanes of $D$ can be partitioned according to which corner-cells $C_{12}, C_{23}, C_{31}$ and/or square ladders $G_1, G_2, G_3$ they cross.

 The number of hyperplanes intersecting corner-cells $C_{(i-1)i}$ and $C_{i(i+1)}$ is bounded by the length of a piece, which is uniformly bounded since $X^*$ is compact. By the minimal complexity of $D$ hyperplanes intersecting $C_{(i-1)i}$ cannot cross $G_{i-1}$ or $G_i$, and the number of hyperplanes intersecting $C_{(i-1)i}$ and crossing $G_{i+1}$ is again bounded by the length of a piece. 

Finally, consider the collection $\mathcal H_{(i-1)i}$ of hyperplanes crossing both $G_{i-1}$ and $G_{i}$. By the minimality of $D$ no two such hyperplanes cross in $D$. Since $X$ is compact, there is a bound $M$ on the number of immersed hyperplanes in $X$. If $|\mathcal H_{(i-1)i}|>M$, then by the pigeon-hole principle $\mathcal H_{(i-1)i}$ contains at least two hyperplanes $U,gU$ that are $\pi_1 X^*$ translates of one another, i.e.\ $g\in \pi_1 X^*$. This yields a contradiction since cutting $D$ along $U$ and $gU$, removing the part between, and gluing back together along $g$, gives a lower complexity disc diagram collared by $W_1, W_2, W_3$.

In the $U_1$-minimal case, we observe that the cutting and gluing procedure preserves $U_1$-minimality, so the above argument is  valid.
\end{proof}

\begin{figure}
\centering
\includegraphics[height=1in]{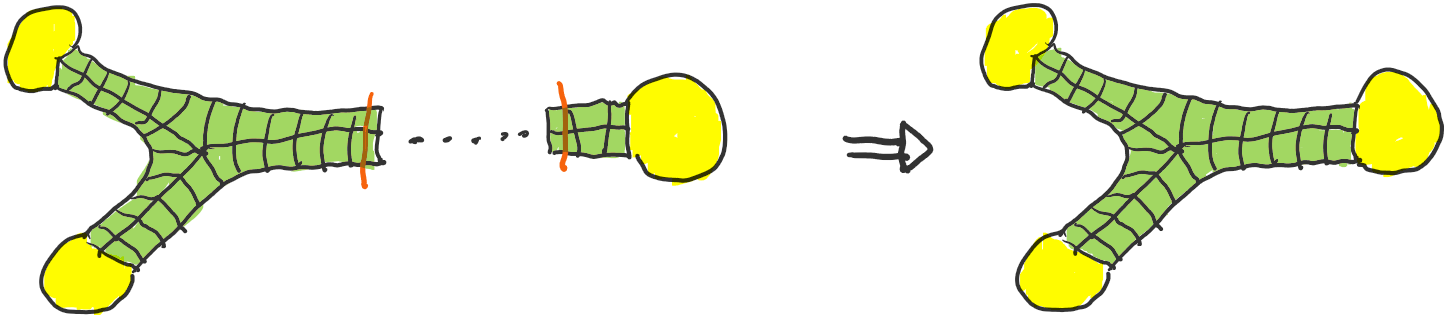}
\caption{Repeated hyperplane allows us to produce smaller $D$.}
\label{fig: walltrianglesproof2}
\end{figure}

We will use the following consequence of Proposition~\ref{prop: uniformly thin wall triangles}.(\ref{part: bounded}). The \emph{$r$-neighborhood} of a wall $Z$ is the union of $r$-neighborhoods of its hyperplanes.
\begin{cor}\label{cor: existence of r}
    Let $X^*=\langle X \mid \{Y_i\} \rangle $ be a compact cubical presentation that satisfies the $C'(\frac{1}{12})$ and  $B(6)$ conditions, and $11$-wall convexity.
    \begin{enumerate}
        \item\label{case: Z case} There exists $r$ such that for every wall $Z$ in $\widetilde X^*$, if any two other walls $W,W'$ intersect $Z$ and each other, then $W,W'$ intersect each other in the $r$-neighborhood $\neb_r(Z)\subseteq \widetilde X^*$ of $Z$.
        \item\label{case: U case} There exists $r'$ such that for every wall $Z$ in $\widetilde X^*$ and every hyperplane $U$ of $Z$, if any two other walls $W,W'$ intersect each other and intersect $U$ in hyperplanes $V, V'$ respectively, then $V\cap N(U),V'\cap N(U)$ are at distance $r'$ in $N(U)$.
    \end{enumerate}
    
\end{cor}

%%%%%%%%%%%%%%%%%%%%%%%%%%%%%%%%%%%%%%%%%%%%%%%%%%%%%%%%%%%%%%%%%%%%%
\subsection{Hemiwallspaces}\label{subsec: hemi}
%%%%%%%%%%%%%%%%%%%%%%%%%%%%%%%%%%%%%%%%%%%%%%%%%%%%%%%%%%%%%%%%%%%%%

For a wallspace $(X, \mathcal{W})$, an associated \emph{hemiwallspace} is a subcollection $\Hem$ of halfspaces of $\mathcal{W}$ such that for each wall $W=\{W^+,W^-\}$ at least one of  $W^+,W^-$ is contained in $\Hem$. Every hemiwallspace has a dual, which admits a natural embedding as a convex subcomplex
of the dual of $(X, \mathcal{W})$.
 See \cite[Sec~3.4]{HruskaWiseAxioms} for more about hemiwallspaces.
For a subset $S\subseteq X$, the collection of all halfspaces intersecting $S$ nontrivially is a hemiwallspace, which we denote by $\Hem(S)$. The dual of $\Hem(S)$ is denoted by $\C(\Hem(S))$.

Let $Z$ be a wall in $\widetilde X^*$. 
 For a set $S\subseteq \widetilde X^*$, the collection of halfspaces of walls crossing $Z$ that also intersect $S$ is equal to $\Hem(Z)\cap \Hem(S)$, which is still a hemiwallspace.

\begin{lem}\label{lem:hyperplane carrier}
Let $W,Z\in \mathcal W$ be two crossing walls of a wallspace $(X,\mathcal W)$, and let $\C$ be the convex subcomplex of the dual $\C(\Hem(Z))$, corresponding to $\Hem(Z)\cap\Hem(W)$.
Then $\C$ is a hyperplane carrier in $\C(\Hem(Z))$
\end{lem}
\begin{proof}
Ignore the halfspaces of $\Hem(Z)\cap \Hem(W)$ whose complements are not in $\Hem(Z)\cap \Hem(W)$, as these halfspaces only play a role in situating $\C\subset \C(\Hem(Z))$. 
Let $V$ be the hyperplane of $\C(\Hem(Z))$ corresponding to $W$.
There is a natural bijection between halfspaces of $N(V)\subseteq \C(\Hem(Z))$ and halfspaces of $\Hem(Z)\cap \Hem(W)$, and it preserves non-empty intersection. Thus the duals are isomorphic.
\end{proof}

Lemma~\ref{lem:hyperplane carrier} will be applied to two walls in a cubical presentation $X^*$ satisfying  the $B(6)$ condition.

%%%%%%%%%%%%%%%%%%%%%%%%%%%%%%%%%%%%%%%%%%%%%%%%%%%%%%%%%%%%%%%%%%%%%
\subsection{Relative hyperbolicity and local relative quasiconvexity}\label{sec: _relative hyperbolicty}
%%%%%%%%%%%%%%%%%%%%%%%%%%%%%%%%%%%%%%%%%%%%%%%%%%%%%%%%%%%%%%%%%%%%%

We follow the approach to  relative hyperbolicity using Bowditch's fine  hyperbolic graphs \cite{Bowditch2012}.
Recall that $G$ is \emph{hyperbolic relative to $\{G_v\}$} if $G$ acts cocompactly on a fine hyperbolic graph $\Gamma$ with finite edge stabilizers, each $G_v$ stabilizes a vertex, and each vertex stabilizer is either finite or conjugate to some $G_v$.
In this setting, the condition that $H$ is  \emph{relatively quasiconvex} is that there is an $H$-cocompact quasiconvex subgraph $\Upsilon\subset \Gamma$, in some (and in fact any) action of $G$ on such a fine hyperbolic graph $\Gamma$ \cite{MartinezPedrozaWise2011}.

\begin{lem}\label{lem:Gsplits}
    Let $G$ split as a finite graph of groups with finite edge groups, then $G$ is hyperbolic relative to its (possibly infinite) vertex groups.
\end{lem}

\begin{proof}
Consider the action of $G$ on the associated Bass-Serre tree $T$.
Note that $T$ is obviously $0$-hyperbolic and fine.
The quotient is compact, and the action has finite edge stabilizers. Thus by Bowditch's criterion, $G$ is hyperbolic relative to its vertex stabilizers.
\end{proof}

The following generalizes the local quasiconvexity of free groups:
\begin{lem}[Local relative quasiconvexity]\label{lem:LRHQC}
Let 
$G$ split as a finite graph of groups with finite edge groups and vertex groups $\{G_v\}$. 
So $G$ is hyperbolic relative to $\{G_v\}$ by Lemma~\ref{lem:Gsplits}.

Let $H$ be a subgroup generated by a finite set, together with finitely many subgroups of conjugates of the $\{G_v\}$. Then $H$ is relatively quasiconvex.

In particular, $G$ is locally relatively quasiconvex.
\end{lem}
In our application, $G$ splits as a 
finite bipartite graph of groups where the left vertex groups are finite and the right vertex groups are arbitrary.

\begin{proof}
Let $T$ be the Bass-Serre tree associated to the finite graph of groups for $G$. Let $H$ be a subgroup generated by a finite set of  elements $\{g_1, \dots, g_k\}$ and a finite set of elliptic subgroups $\{K_{v_1}, \dots, K_{v_\ell}\}$, i.e.\ $K_{v_j}$ is a subgroup of 
the stabilizer of a vertex $v_j$ in $T$. 

Without loss of generality, we assume each $g_i$ is loxodromic. For $1\leq i\leq k$, let $\gamma_i$ be a fundamental domain for the action of $\langle g_i\rangle$ on its axis. 

Let $J$ be a finite tree containing each $\gamma_i$ and each $v_i$.
Then $\Upsilon=HJ$ is an $H$-cocompact subtree of $T$. Finally, $\Upsilon\subset T$ is obviously quasiconvex.
\end{proof}

%%%%%%%%%%%%%%%%%%%%%%%%%%%%%%%%%%%%%%%%%%%%%%%%%%%%%%%%%%%%%%%%%%%%%
\subsection{Relative cocompactness}\label{sec: relative hyperbolicty}
%%%%%%%%%%%%%%%%%%%%%%%%%%%%%%%%%%%%%%%%%%%%%%%%%%%%%%%%%%%%%%%%%%%%%

\begin{lem}[Relatively-hyperbolic wall-stabilizers]\label{lem: wall stabilizer relatively hyperbolic}
    Let $X^*$ satisfy  $B(6)$. For a wall $Z$ in $\widetilde X^*$, 
    the group $\Stab(Z)$ acts on a bipartite tree whose vertices correspond to $Z$-cones and $Z$-hyperplanes, and whose edges correspond to incidence. 

    In particular, $\Stab(Z)$ is hyperbolic relative to $\{\Stab(U_i)\}_i$ where $U_i$ are the hyperplanes of $Z$.
\end{lem}
\begin{proof}
The graph of groups whose vertices are stabilizers of hyperplanes and cone-cells of $Z$ is defined \cite[Def~5.16]{WiseBook}, and shown to be a tree in \cite[Thm~5.17/5.20]{WiseBook}.
%defined in 5.16, and shown to be a tree in 5.17/5.20.
The result then holds by Lemma~\ref{lem:Gsplits}. Indeed, $\Aut_X(Y)$ is finite since $Y$ is compact. Thus the stabilizer of the cone over $Y$ is finite in $\pi_1 X^*$ as a quotient of $\Aut_X(Y)$, and the stabilizer of any cone-cell is contained in the stabilizer of the cone over $Y$.
\end{proof}

Sageev proved cocompactness of the dual when the wallspace is cocompact, the group is hyperbolic, and the wall-stabilizers are quasiconvex  \cite{Sageev97}.
We use the following generalization %proven in 
\cite[Thm~7.12]{HruskaWiseAxioms}:

\begin{prop}[Relative Cocompactness]\label{prop:RelHyp relative Cocompactness} Let $(X, \mathcal{W})$ be a wallspace such that $X$ is also a length space. 
Let $G$ act properly and cocompactly on $X$ preserving both its metric and wallspace structures. Suppose $G$ is hyperbolic relative to $\mathbb{P}$, and for each $P \in \mathbb{P}$ let $Y = Y(P) \subset X$ be a non-empty $P$-invariant $P$-cocompact subspace.
Suppose that the action on $\mathcal{W}$ has only finitely many $G$-orbits of walls, and $\text{Stab}(W)$ is relatively quasiconvex and acts cocompactly on $W$ for each $W\in \mathcal W$.
Then there exists a compact $K$ with $GK$ connected such that $\mathcal{C}(X) = GK \cup_{PK} G\mathcal{C}(\Hem(Y))$.
In particular, $\C(X)$ is $G$-cocompact, if $\C(\Hem(Y))$ is $P$-cocompact for each $P\in \mathbb P$ and $Y=Y(P)$.
\end{prop}

%%%%%%%%%%%%%%%%%%%%%%%%%%%%%%%%%%%%%%%%%%%%%%%%%%%%%%%%%%%%%%%%%%%%%
\subsection{Proof of cocompactness}\label{subsec:proof of cocompactness}
%%%%%%%%%%%%%%%%%%%%%%%%%%%%%%%%%%%%%%%%

\begin{lem}\label{lem:uniform degree}
    Let $\C$ be a $G$-cocompact CAT(0) cube complex with finite vertex stabilizers. There is a uniform upper bound on degrees of vertices of $\C$.
Thus each vertex of $\C$ lies in uniformly finitely many hyperplane carriers.
\end{lem}

%\begin{rem}
%    In the statement of Lemma~\ref{lem:uniform degree} we denote a cube complex by $\C$, not $X$ as earlier, because in its application in the proof of Theorem~\ref{thm:cocompactness} the lemma applies to certain subcomplexes of $X$, where $X^*=\langle X \mid \{Y_i\}\rangle $.
%\end{rem}

\begin{proof}[Proof of Lemma~\ref{lem:uniform degree}] It suffices to prove the statement for a group $G$ acting on a graph $\Gamma$ (the $1$-skeleton of $\C$).
For each vertex $v$ of $\Gamma$,
let $e_1,\dots, e_n$ be $\Stab(v)$-representatives of the edges at $v$.
Then $\deg(v)=\sum_i \Stab(v)/\Stab(e_i)$. The statement follows by using $\max\{\deg(v)\}$ as $v$ ranges over the finitely many $G$-orbits of vertices in $\Gamma$.

If $v\in N(U)$
then $U$ is dual to an edge at $v$.
Hence $\deg(v)$ is uniformly bounded as above.
\end{proof}

The proof of Theorem~\ref{thm:cocompactness} uses the following %equivalent 
cocompactness
characterization.

\begin{rem}\label{rem:cocompactness}
Let $\C$ be the CAT(0) cube complex dual to a wallspace with a $G$-action.
Then the following statements are equivalent
\begin{enumerate}
    \item\label{rmk:characerization cocompact 1} $\C$ is $G$-cocompact,
\item\label{rmk:characerization cocompact 2}  there are finitely many $G$-orbits  of maximal cubes in $\C$, 
\item\label{rmk:characerization cocompact 3}  there are finitely many $G$-orbits of collections of pairwise crossing walls in the wallspace.
\end{enumerate}
\end{rem}

\begin{lem}\label{lem: stab-cocompactness}
Let $\widetilde X$ be a $G$-cocompact CAT(0) cube complex, and $ U\subseteq \widetilde X$ a hyperplane. Then $U$ is $\Stab_G(U)$-cocompact. 
\end{lem}

\begin{proof}
    We use characterization~\eqref{rmk:characerization cocompact 2} of cocompact action from Remark~\ref{rem:cocompactness}.
    Since $G$ acts cocompactly on $X$, there are finitely many $G$-orbits of midcubes in $\widetilde U$ under the action of $G$. We claim that each such orbit is invariant under the action of $\Stab_G(\widetilde U)$. Indeed, if $g\in G$ sends one midcube of $\widetilde U$ to another, then it must stabilize $\widetilde U$, since $\widetilde U$ is uniquely determined by any of its midcubes, and $g$ sends hyperplanes to hyperplanes as a cubical automorphism.
\end{proof}

\begin{proof}[Proof of Theorem~\ref{thm:cocompactness}]
We  use Remark~\ref{rem:cocompactness}.\eqref{rmk:characerization cocompact 3}. 
Throughout this proof,  we write $\Stab(\cdot)$ for $\Stab_{\pi_1(X^*)}(\cdot)$.

We will show that for each wall $Z$  in $\widetilde X^*$, there are finitely many  $\Stab(Z)$-orbits of collections of pairwise-crossing walls that include $Z$.
Since $\widetilde X^*$ is cocompact, there are finitely many orbits of walls in $\widetilde X^*$.
Thus we can conclude that there are finitely many orbits of collections of pairwise-crossing walls.

We will first show that the dual $\C(\Hem(Z))$ of the hemiwallspace $\Hem(Z)$ is $\Stab(Z)$-cocompact.
Let $U$ be a hyperplane of $Z$. Note that $U$ is embedded and simply connected by Lemma~\ref{lem:Greendlinger}. See~\cite[Lem~3.11(i)+(ii)]{Arenas24asph} for details.
Since $X$ is compact, $U$ is $\Stab( U)$-cocompact by Lemma~\ref{lem: stab-cocompactness}.

We first prove that the dual 
$\C(\Hem(U))$
is $\Stab(U)$-cocompact. 
Let $W, W'\in \Hem(U)$ be intersecting walls that cross $U$,
and let $V, V'$ be the hyperplanes in $\widetilde X^*$ contained in $W,W'$ respectively, which cross $U$.
Let $r'\in \mathbb N$ satisfy Corollary~\ref{cor: existence of r}.\eqref{case: U case}.
 Let $\C$ and $\C'$ be the subcomplexes of $N(U)$ that are the \emph{$r'$-thickened} carriers of the hyperplanes $V\cap N(U), V'\cap N(U)$ of $N(U)$. These are convex subcomplexes containing $r'$-neighborhoods of the hyperplane carriers of $V\cap N(U), V'\cap N(U)$ \cite{HaglundWiseCombination}.
 By Corollary~\ref{cor: existence of r}.\eqref{case: U case}, the hyperplanes $V,V'$ are at distance at most $r'$, so in particular $\C\cap \C' \neq \emptyset$.
Thus every collection $\{W_i\}_{i \in I}$ of $U$-crossing, pairwise-crossing walls in $\Hem(U)$
corresponds to a collection of pairwise-intersecting $r'$-thickened hyperplane carriers $\{\C_i\}_{i \in I}$ in $N(U)$.
In particular, for every finite collection $\{W_i\}_{i \in I}$ of $U$-crossing, pairwise-crossing walls, $\bigcap_{i
\in I} \C_i \neq \emptyset$ by Lemma~\ref{lem:helly}. 
There is a uniform bound on the number of $\Stab(U)$-orbits of collections of $r'$-thickened hyperplane carriers containing a $0$-cube, by the cocompactness of $U$ (and hence of $N(U)$). 
By compactness of $X$ there is a uniform upper bound on the number of $r'$-thickened hyperplane carriers containing any given $0$-cube of $X$, which implies that there is a uniform upper bound on the size of a collection of pairwise-intersecting $r'$-thickened hyperplane carriers in $X$, hence also in $N(U)$.
This implies that $\C(\Hem(U))$ is $\Stab(U)$-cocompact.

Let $Z$ be a wall which is a union of hyperplanes $\{U_i\}$.
Then $\Stab(Z)$ is hyperbolic relative to $\{\Stab(U_i)\}$ by Lemma~\ref{lem: wall stabilizer relatively hyperbolic}.
Let $r\in \mathbb N$ be provided by Corollary~\ref{cor: existence of r}.\eqref{case: Z case}, i.e.\ for walls $W$ and $W'$ that cross each other and cross $Z$, we have that $W, W'$ cross each in $\neb_r(Z)$. 
By compactness of $X^*$, the action of $\Stab(Z)$ on $\neb_r(Z)$ is proper and cocompact, and has finitely many orbits of walls from $\Hem(Z)$.

We now apply Proposition~\ref{prop:RelHyp relative Cocompactness} with $G = \Stab(Z)$,  $(X, \mathcal W)=(\neb_r(Z), \Hem(Z))$, $\mathbb{P} = \{\Stab(U_i)\}$, and $U_i$ playing the role of subspaces $Y$.
To view $\Hem(Z)$ as a wallspace, we ignore the halfspaces whose complements are not contained in $\Hem(Z)$. 
For each $W\in \Hem(Z)$, the stabilizer $\Stab_{\Stab(Z)}(W)$ is relatively-quasiconvex in $\Stab(Z)$ by Lemma~\ref{lem:LRHQC}, and $W$ is $\Stab_{\Stab(Z)}(W)$-cocompact by compactness of $X^*$.
Thus, $\C(\Hem(Z))$ is  $\Stab(Z)$-cocompact by Proposition~\ref{prop:RelHyp relative Cocompactness}.

To complete the proof, we show that $\Stab(Z)$-cocompactness of $\C(\Hem(Z))$ implies the finiteness of the $\Stab(Z)$-orbits of collections of pairwise-crossing walls that include $Z$. The argument is similar to the earlier argument showing the $\Stab(U)$-cocompactness of $\C(\Hem(U))$, we include it for completeness.

Let $W, W'$ be intersecting walls that cross $Z$, and let $\C, \C'$ be the convex subcomplexes of the dual $\C(\Hem(Z))$, corresponding to $\Hem(Z)\cap \Hem(W)$ and $\Hem(Z)\cap \Hem(W')$, and  which are hyperplane carriers in $\C(\Hem(Z))$ by Lemma~\ref{lem:hyperplane carrier}. Then $\C\cap \C' \neq \emptyset$, as otherwise there is a hyperplane $V$ in $\C(\Hem(Z))$ that separates $\C$ from $\C'$, which is dual to a wall that separates $W$ from $W'$, contradicting that $W, W'$ intersect. Thus, every collection $\{W_i\}_{i \in I}$ of $Z$-crossing, pairwise-crossing walls corresponds to the pairwise-intersecting collection $\{\C_i\}_{i \in I}$. If $I$ is finite, then $\bigcap_{i
\in I} \C_i \neq \emptyset$ by Lemma~\ref{lem:helly}. 
Each $0$-cube lies in finitely many $\Stab(Z)$-orbits of hyperplane carriers in $\C(\Hem(Z))$ by Lemma~\ref{lem:uniform degree}.  By $\Stab(Z)$-cocompactness of $\C(\Hem(Z))$, there are finitely many $\Stab(Z)$-orbits of $0$-cubes in $\C(\Hem(Z))$. Thus each collection of pairwise-crossing walls corresponds to one of these finitely many $\Stab(Z)$-orbits of collections of mutually intersecting hyperplane carriers.
\end{proof}

\bibliographystyle{alpha}
\bibliography{sample}

\def\cprime{$'$} \def\cprime{$'$}
\begin{thebibliography}{MPW11}

\bibitem[Ago13]{AgolGrovesManning2012}
Ian Agol.
\newblock The virtual {H}aken conjecture.
\newblock {\em Doc. Math.}, 18:1045--1087, 2013.
\newblock With an appendix by Agol, Daniel Groves, and Jason Manning.

\bibitem[AH22]{ArzhantsevaHagen16}
Goulnara~N. Arzhantseva and Mark~F. Hagen.
\newblock Acylindrical hyperbolicity of cubical small cancellation groups.
\newblock {\em Algebr. Geom. Topol.}, 22(5):2007--2078, 2022.

\bibitem[AJW24]{ArenasJankiewiczWise}
Macarena Arenas, Kasia Jankiewicz, and Daniel~T. Wise.
\newblock Hyperbolicity in non-metric cubical small-cancellation.
\newblock {\em Bulletin of the London Mathematical Society}, 56(6):2036--2052,
  2024.

\bibitem[Are24a]{Arenas24IMRN}
Macarena Arenas.
\newblock Asphericity of cubical presentations: the 2-dimensional case.
\newblock {\em Int. Math. Res. Not.}, 2024(7):5524--5547, 2024.

\bibitem[Are24b]{Arenas24asph}
Macarena Arenas.
\newblock Asphericity of cubical presentations: the general case, 2024.

\bibitem[Are24c]{ArenasAGT}
Macarena Arenas.
\newblock A cubical {R}ips construction.
\newblock {\em Algebr. Geom. Topol.}, 24(8):4353--4372, 2024.

\bibitem[BHS17]{BehrstockHagenSisto}
Jason Behrstock, Mark~F. Hagen, and Alessandro Sisto.
\newblock Hierarchically hyperbolic spaces, {I}: {C}urve complexes for cubical
  groups.
\newblock {\em Geom. Topol.}, 21(3):1731--1804, 2017.

\bibitem[BM97]{BurgerMozes97}
Marc Burger and Shahar Mozes.
\newblock Finitely presented simple groups and products of trees.
\newblock {\em C. R. Acad. Sci. Paris S\'er. I Math.}, 324(7):747--752, 1997.

\bibitem[Bow12]{Bowditch2012}
B.~H. Bowditch.
\newblock Relatively hyperbolic groups.
\newblock {\em Internat. J. Algebra Comput.}, 22(3):1250016, 66, 2012.

\bibitem[BSV14]{GGTbook14}
Mladen Bestvina, Michah Sageev, and Karen Vogtmann, editors.
\newblock {\em Geometric group theory}, volume~21 of {\em IAS/Park City
  Mathematics Series}.
\newblock American Mathematical Society, Providence, RI; Institute for Advanced
  Study (IAS), Princeton, NJ, 2014.

\bibitem[CS11]{CapraceSageev2011}
Pierre-Emmanuel Caprace and Michah Sageev.
\newblock Rank rigidity for {CAT(0)} cube complexes.
\newblock {\em Geometric And Functional Analysis}, 21:851--891, 2011.
\newblock 10.1007/s00039-011-0126-7.

\bibitem[DFW19]{DahmaniFuterWise}
Fran\c{c}ois Dahmani, David Futer, and Daniel~T. Wise.
\newblock Growth of quasiconvex subgroups.
\newblock {\em Math. Proc. Cambridge Philos. Soc.}, 167(3):505--530, 2019.

\bibitem[DGO17]{DahmaniGuirardelOsin2017}
F.~Dahmani, V.~Guirardel, and D.~Osin.
\newblock Hyperbolically embedded subgroups and rotating families in groups
  acting on hyperbolic spaces.
\newblock {\em Mem. Amer. Math. Soc.}, 245(1156):v+152, 2017.

\bibitem[FW24]{FuterWise}
David Futer and Daniel~T. Wise.
\newblock Cubulating random quotients of hyperbolic cubulated groups.
\newblock {\em Trans. Amer. Math. Soc. Ser. B}, 11:622--666, 2024.

\bibitem[Gen25]{Genevois25}
Anthony Genevois.
\newblock Cyclic hyperbolicity in {CAT(0)} cube complexes, 2025.

\bibitem[Ger98]{Gerasimov98}
V.~Gerasimov.
\newblock Fixed-point-free actions on cubings [translation of {\it {a}lgebra,
  geometry, analysis and mathematical physics ({r}ussian) ({n}ovosibirsk,
  1996)}, 91--109, 190, {I}zdat. {R}oss. {A}kad. {N}auk {S}ibirsk. {O}tdel.
  {I}nst. {M}at., {N}ovosibirsk, 1997; {MR}1624115 (99c:20049)].
\newblock {\em Siberian Adv. Math.}, 8(3):36--58, 1998.

\bibitem[Hag08]{HaglundGraphProduct}
Fr{\'e}d{\'e}ric Haglund.
\newblock Finite index subgroups of graph products.
\newblock {\em Geom. Dedicata}, 135:167--209, 2008.

\bibitem[Hul16]{Hull2016}
Michael Hull.
\newblock Small cancellation in acylindrically hyperbolic groups.
\newblock {\em Groups Geom. Dyn.}, 10(4):1077--1119, 2016.

\bibitem[HW12]{HaglundWiseCombination}
Fr\'{e}d\'{e}ric Haglund and Daniel~T. Wise.
\newblock A combination theorem for special cube complexes.
\newblock {\em Ann. of Math. (2)}, 176(3):1427--1482, 2012.

\bibitem[HW14]{HruskaWiseAxioms}
G.~C. Hruska and Daniel~T. Wise.
\newblock Finiteness properties of cubulated groups.
\newblock {\em Compos. Math.}, 150(3):453--506, 2014.

\bibitem[HW24]{HuangWiseSpecial}
Jingyin Huang and Daniel~T. Wise.
\newblock Virtual specialness of certain graphs of special cube complexes.
\newblock {\em Math. Ann.}, 388(1):329--357, 2024.

\bibitem[Jan17]{Jankiewicz2017}
Kasia Jankiewicz.
\newblock The fundamental theorem of cubical small cancellation theory.
\newblock {\em Trans. Amer. Math. Soc.}, 369(6):4311--4346, 2017.

\bibitem[JW22]{JankiewiczWise18}
Kasia Jankiewicz and Daniel~T. Wise.
\newblock Cubulating small cancellation free products.
\newblock {\em Indiana Univ. Math. J.}, 71(4):1397--1409, 2022.

\bibitem[Kap99]{Kapovich99}
Ilya Kapovich.
\newblock A non-quasiconvexity embedding theorem for hyperbolic groups.
\newblock {\em Math. Proc. Cambridge Philos. Soc.}, 127(3):461--486, 1999.

\bibitem[MPW11]{MartinezPedrozaWise2011}
Eduardo Mart{\'{\i}}nez-Pedroza and Daniel~T. Wise.
\newblock Relative quasiconvexity using fine hyperbolic graphs.
\newblock {\em Algebr. Geom. Topol.}, 11(1):477--501, 2011.

\bibitem[Osi16]{Osin2016}
D.~Osin.
\newblock Acylindrically hyperbolic groups.
\newblock {\em Trans. Amer. Math. Soc.}, 368(2):851--888, 2016.

\bibitem[Sag97]{Sageev97}
Michah Sageev.
\newblock Codimension-$1$ subgroups and splittings of groups.
\newblock {\em J. Algebra}, 189(2):377--389, 1997.

\bibitem[Sta83]{Stallings83}
John~R. Stallings.
\newblock Topology of finite graphs.
\newblock {\em Invent. Math.}, 71(3):551--565, 1983.

\bibitem[SW15]{SageevWiseCores}
Michah Sageev and Daniel~T. Wise.
\newblock Cores for quasiconvex actions.
\newblock {\em Proc. Amer. Math. Soc.}, 143(7):2731--2741, 2015.

\bibitem[Wis01]{WisePositive}
Daniel~T. Wise.
\newblock The residual finiteness of positive one-relator groups.
\newblock {\em Comment. Math. Helv.}, 76(2):314--338, 2001.

\bibitem[Wis12]{WiseCBMS2012}
Daniel~T. Wise.
\newblock {\em From riches to raags: 3-manifolds, right-angled {A}rtin groups,
  and cubical geometry}, volume 117 of {\em CBMS Regional Conference Series in
  Mathematics}.
\newblock Published for the Conference Board of the Mathematical Sciences,
  Washington, DC, 2012.

\bibitem[Wis21]{WiseBook}
Daniel~T. Wise.
\newblock {\em The structure of groups with a quasiconvex hierarchy}, volume
  209 of {\em Annals of Mathematics Studies}.
\newblock Princeton University Press, Princeton, NJ, [2021] \copyright 2021.

\end{thebibliography}

\end{document}